\documentclass[reqno]{amsart}

\usepackage[bottom=3cm, top=3cm, left=3cm, right=3cm]{geometry}
\usepackage{amsmath}
\usepackage{amsthm}
\usepackage{amssymb}
\usepackage{etoolbox}
\usepackage{mathtools}
\usepackage{cases}
\usepackage{empheq}
\usepackage[shortlabels]{enumitem}
\setlist[itemize]{noitemsep,nolistsep}
\setlist[enumerate]{noitemsep,nolistsep}

\usepackage[dvipsnames]{xcolor}
\usepackage{graphicx}
\definecolor{tabblue}{rgb}{0.12156862745098039, 0.4666666666666667, 0.7058823529411765}
\definecolor{taborange}{rgb}{1.0, 0.4980392156862745, 0.054901960784313725}
\definecolor{tabgreen}{rgb}{0.17254901960784313, 0.6274509803921569, 0.17254901960784313}
\definecolor{tabred}{rgb}{0.8392156862745098, 0.15294117647058825, 0.1568627450980392}
\definecolor{tabpurple}{rgb}{0.5803921568627451, 0.403921568627451, 0.7411764705882353}

\colorlet{citecolor}{tabgreen}
\colorlet{linkcolor}{tabblue}
\colorlet{urlcolor}{taborange!75!black}
\usepackage{tikz}
\usepackage{tikz-cd}

\usepackage{hyperref}
\hypersetup{
  breaklinks=true,
  colorlinks=true,
  linkcolor=linkcolor,
  urlcolor=urlcolor,
  citecolor=citecolor,
  }
  
\usepackage[backref=true, backrefstyle=three, hyperref=true, maxcitenames=2, maxbibnames=50, style=alphabetic,backend=biber]{biblatex}
\usepackage[noabbrev,capitalize,nameinlink]{cleveref}
\crefname{equation}{}{}
\Crefname{equation}{Eq.}{}

\usepackage{thmtools}
\usepackage{framed}

\input{commands}
\usepackage[colorinlistoftodos,prependcaption,textsize=tiny,textwidth=2.6cm]{todonotes}

\makeatletter
\@mparswitchfalse%
\makeatother
\normalmarginpar
\makeatletter
\providecommand\@dotsep{5}
\renewcommand{\listoftodos}[1][\@todonotes@todolistname]{%
  \@starttoc{tdo}{#1}}
\makeatother

\makeatletter
\renewcommand{\tocsection}[3]{%
  \indentlabel{\@ifnotempty{#2}{\bfseries\ignorespaces#1 #2\quad}}\bfseries#3}
\renewcommand{\tocsubsection}[3]{%
  \indentlabel{\@ifnotempty{#2}{\ignorespaces#1 #2\quad}}#3}
\def\l@subsection{\@tocline{2}{0pt}{2.5pc}{5pc}{}}
\makeatother

\allowdisplaybreaks

\newcounter{counter}
\numberwithin{counter}{section}
\newtheorem{theorem}[counter]{Theorem}
\newtheorem{lemma}[counter]{Lemma}
\newtheorem{proposition}[counter]{Proposition}
\newtheorem{corollary}[counter]{Corollary}
\theoremstyle{definition}
\newtheorem{remark}[counter]{Remark}
\newtheorem{example}[counter]{Example}
\newtheorem{definition}[counter]{Definition}

\numberwithin{equation}{section}
\numberwithin{equation}{section}

\title[RIESZ KERNEL WASSERSTEIN GRADIENT FLOWS]{
    \Large{On the global convergence of Wasserstein gradient flow of the Coulomb discrepancy.}}
\author{Siwan Boufadene$^\dagger$}

\author{François-Xavier Vialard$^\dagger$}

\begin{document}

\maketitle

\vspace{-8mm}
\begin{abstract}
    In this work, we study the Wasserstein gradient flow of the Riesz energy towards a determined target measure on the space of probability measures. 
    The Riesz energy is a quadratic functional on this space, defined using Riesz kernels, and it is in general not geodesically convex in the Wasserstein geometry. Consequently, standard arguments cannot be applied to deduce the global convergence of the Wasserstein gradient flow. 
    Our main result is the exponential convergence of the flow to the minimizer on a closed Riemannian manifold under the condition that the logarithms of the source and target measures are bounded and H\"older continuous. 
    To show this, we first prove that a Polyak-Lojasiewicz inequality is satisfied for sufficiently regular solutions. The key regularity result is the global-in-time existence of H\"older solutions if the initial and target data are  H\"older continuous, proven either in Euclidean spaces or in closed Riemannian manifolds. 
    We then define Lagrangian critical points and prove that such points, for the Coulomb or Energy distance discrepancies, are equal to the target everywhere except on singular sets with empty interiors.
    For arbitrary measures, we use flow interchange techniques to prove there are no local minima other than the global one for the Coulomb kernel. 
    Additionally, sufficiently singular measures cannot be critical points, ensuring they are not fixed points of the discrete JKO updates. 
\end{abstract}

\vspace{1mm}
\noindent\textbf{Keywords.} Calculus of variations $\cdot$ Optimal transportation $\cdot$ Coulomb kernel $\cdot$ Flow interchange

\vspace{2mm}
\noindent\textbf{Mathematics Subject Classification.} 49Q22, 49Q10, 90C26. 

\vspace*{\fill}

Email addresses: \href{mailto:siwan.boufadene@univ-eiffel.fr}{siwan.boufadene@univ-eiffel.fr} and \href{mailto:francois-xavier.vialard@u-pem.fr}{francois-xavier.vialard@univ-eiffel.fr}

\hypersetup{linkcolor=black}
\newpage
\setcounter{tocdepth}{3}
{
    \hypersetup{linkcolor=black}
    \tableofcontents
}

\newpage

\allowdisplaybreaks
\section{Introduction}
\label{sec:introduction}
In this paper, we are interested in the Wasserstein gradient flow of some quadratic functionals defined on the set of probability measures, both on the Euclidean space and on a Riemannian manifold. These quadratic functionals are convex for the standard vertical convex structure of probability measures. However they are not geodesically convex for the Wasserstein geometry, potentially impacting the global convergence properties typically obtained in geodesically convex scenarios. The motivation for studying these gradient flows under the Wasserstein geometry comes from machine learning, more precisely from the mean-field limit of shallow neural networks \cite{chizat2018global,Mei_2018}. This line of research investigates the optimization landscape of the usual empirical risk of a single-hidden layer neural network under gradient flow. A powerful relaxation of the problem, already proposed in \cite{BarronUniversal1993}, consists of embedding the space of parameters into the space of probability measures. In this context, the corresponding objective functional takes on a quadratic form and the particle gradient flow corresponds to a Wasserstein gradient flow. 
The Coulomb MMD energy is a particular case of quadratic functionals defined by reproducing kernels. These functionals have also raised interest in machine learning and statistics since they yield a discrepancy between probability measures and, in fact, a squared distance. These Maximum Mean Discrepancies (MMD) possess two noteworthy properties: (i) a quadratic computational complexity (or even $O(n\log(n))$ for the Energy distance kernel \cite{hertrich2023generative}), outperforming alternatives like optimal transport, and (ii) a parametric estimation rate from empirical measures, a feature missing from standard optimal transport methods.

Let $G$ be a (conditionally positive) kernel, such as the Gaussian kernel, on the Euclidean space $\mathbb{R}^d$. The MMD between $\mu$ and $\nu$ two probability measures is the energy
\begin{equation*}
    E_\nu(\mu) = \iint (\mu(x) - \nu(x) )G(x,y) (\mu(y) - \nu(y))\,.
\end{equation*}

Such a functional is nonnegative and strictly convex on the space of probability measures if the kernel $G$ is conditionally positive. 
In our work, we are interested in a fixed target measure $\nu$ and a time-dependent $\mu_t$, optimized through the action of velocity fields. More precisely, we are interested in the Wasserstein gradient flow of the functional $E_\nu$ with respect to $\mu$.
Our primary concern is the question of global convergence towards the unique minimizer, which is $\nu$.
For example, when the kernel is smooth enough, empirical measures remain preserved by the Wasserstein gradient flow. Consequently, the Wasserstein gradient flow of this energy with a finite empirical measure as the source and a density as the target cannot give convergence. 
This fact motivates exploring non-smooth kernels, such as the energy distance $(x,y) \mapsto -\| x - y\|$ on the Euclidean space, which is not $C^1$. Consequently, one can hope for global convergence of the Wasserstein gradient flow even when the source measure and the target measure are mutually singular. Indeed, in dimension one, the corresponding functional is geodesically convex in the Wasserstein geometry, implying global convergence of the solution $\mu_t$ towards $\nu$. This specific kernel has been studied in the context of Wasserstein gradient flows in \cite{Hertrich_2024} and \cite{steidl2023} for applications in machine learning and imaging. The authors explicitly leave as open the question of global convergence.

The question we address in this paper is the extension of this one-dimensional result to higher dimensions. There are at least two different directions for generalizing this result. Firstly, the Energy distance kernel $-\| x - y\|$  remains conditionally positive definite on $\mathbb R^d$ for $d\geq 1$. Secondly, in dimension one, $-| x - y|$ is the Coulomb kernel, proportional to the inverse of the Laplacian operator. The inverse of the Laplacian can also be defined in higher dimensions, for instance on $\mathbb R^3$, it is given by $\frac{1}{\| x - y\|}$.
Both kernels belong to the family of Riesz kernels. One motivation for using these kernels also stems from numerical experiments, where the energy distance notably stands out.
Indeed, the energy distance kernel is easy to implement. An efficient and fast method for calculating the discrepancy associated with it, as described in \cite{hertrich2023generative}, allows for large-scale applications. Furthermore,  it exhibits favorable behavior compared to other kernels such as the Gaussian kernel. More precisely, global convergence is observed. Conversely, the Coulomb kernel is more intricate to implement due to its blow-up along the diagonal. This makes drawing conclusions from numerical experiments more delicate.
Yet, from a theoretical point of view, the Coulomb kernel has been studied intensively in the mathematical literature, in particular due to its physical significance \cite{Serfaty2016LargeSW,serfaty2015coulomb}. Recent results presented in \cite{Jabin_2018} and \cite{decourcel2023sharp} study large stochastic systems of interacting particles under Coulomb interaction in the Euclidean space, proving propagation of chaos and convergence to the limit continuity equation using relative entropy methods. In \cite{decourcel2023sharp} this work is done on the torus manifold.

For the Coulomb kernel, given smoothness assumptions on $\mu$ a time-dependent density and $\nu$ a fixed density, the Wasserstein gradient flow associated with $E_\nu$ reads, on $\mathbb R^d$ or on a Riemannian manifold:
\begin{equation}\label{heatflow}
\partial_t \mu = -\nabla \cdot {\big( \mu \nabla \varphi\big)}, \qquad \Delta \varphi = \mu - \nu\,.
\end{equation}
The potential $\varphi$ is a solution to the Poisson equation with source term $\mu - \nu$.
The potential $\varphi$ can be expressed, up to a positive constant, as $\varphi = - G \star (\mu - \nu)$, giving an example of non-linear non-local interactions. Non-local interaction energy systems associated with a radial kernel $W(x) = w(\|x\|)$ are solutions to the equation:
\begin{equation*}
\frac{\partial \mu}{\partial t} = -\nabla \cdot {\big( \mu (\nabla W \star \mu)\big)}\,.
\end{equation*}
Confinement results for these dynamics have been established, depending on the choice of $W$. In many instances, $W$ is assumed to be $\lambda$-convex, preserving particles as discussed in \cite{carrillo-figalli-laurent}, allowing for the use of mean-field techniques. Attractive potentials---e.g. $w'(r) \geq 0$ everywhere---are the simplest ones,  where in some cases, the total mass aggregates at the potential's mass center. Some other potentials, including swarming systems models\cite{Carrillo_2014}, Morse potentials or characteristic function of sets \cite{Balague,CARRILLO2012550,carrillo-figalli-laurent} are said to be attractive-repulsive. Various hypotheses are taken to establish confinement \cite{carrillo-figalli-laurent,Balague}. Some potentials with a singularity at $0$ have been treated, for example in \cite{Balague} with $W(x) = G(x) + W_a(x)$ where $G$ is the Coulomb kernel $G(x) = \|x\|^{-d+2}$ and $W_a$ is an attractive potential that satisfies $\underset{r \rightarrow \infty}{\lim} \, W_a'(r) r^{1/d} = +\infty$. 
However, these results do not apply to our case of study, since our functional is not  $\lambda$-convex and has diffusive properties. Additionally, the confining part, which depends on a target measure $\nu$, is weaker compared to previously cited papers and, as of now, we were not able to prove mass confinement.

\vspace{3mm}
\noindent
\textbf{Main contributions. }
This paper presents two primary results focusing on the Coulomb kernel. 

The first main result is the global convergence of the solution of the continuity equation \ref{EqHold} towards the target $\nu$, on a closed Riemannian manifold within a smooth initial setting.
At the beginning of Section \ref{SectionPL}, assuming the solution is sufficiently smooth and exists at all times, a simple calculation shows that the Polyak-Lojasiewicz inequality holds at all times on a closed Riemannian manifold. Notably, this observation reduces the question of (exponential) global convergence to a regularity inquiry: whether the solutions exist for all time in adequately smooth functional spaces.
Therefore, we first study the gradient flow of the Coulomb discrepancy in a smooth setting, i.e. the corresponding continuity PDE under certain regularity assumptions. Assuming Hölder continuity for initial and target measures, we establish that solutions exist at all times, with propagation of H\"older regularity. This proves global convergence with an exponential rate of convergence in the case of closed Riemannian manifolds. 

The second main result is Theorem \ref{NoLocMin}, which relates to the landscape of the Coulomb discrepancy on both the Euclidean space and closed Riemannian manifolds. It states that, in the Wasserstein geometry, the energy functional has no local minima apart from the global one. More precisely, we prove that if the current measure differs from the target, there always exists a measure curve (starting at the current one) that is $1/2$-H\"older in Wasserstein along which the energy is strictly decreasing. 
This is done through the use of flow interchange techniques,  specifically leveraging properties of the Boltzmann entropy along the flow.

Among other results, we prove in Section \ref{SecCriticalPoints} that any Lagrangian critical point for the Coulomb and Energy distance kernels is equal to the target measure everywhere except on singular sets with empty interiors. In a similar direction, we prove that if the difference between the current measure and the target has a Minkowski dimension smaller than the ambient one, then it cannot be a critical point of the Wasserstein gradient flow.

\vspace{3mm}
\noindent
\textbf{Perspectives. }
Left open by our work is the question of global convergence of the flow in a closed Riemannian manifold for all source and target measures. Although the result seems highly plausible on a closed Riemannian manifold, extending it to a non-compact setting such as the Euclidean space would require addressing the confinement issue or changing completely the proof strategy. Indeed, there is a competition between the repulsive behavior of the Coulomb kernel and the attraction of the target measure. The repulsive part can lead to mass spreading to infinity in the Euclidean space, which makes the analysis more difficult in our opinion. 
Note also that in a finite-dimensional setting, knowing that there is no local minima but the global one guarantees the global convergence for almost every initial condition under some assumptions on the objective functional. Obtaining similar results in our infinite-dimensional case would be of interest.

\vspace{3mm}
\noindent
\textbf{Notations.}
\begin{itemize}
    \item If $A$ is a subset of $\mathbb{R}^d$, $A^c \coloneqq \mathbb{R}^d \setminus A$ denotes its complement.    
    \item $\mu^{\otimes 2}$ is the product measure on $\mathbb{R}^d \times \mathbb{R}^d$: for any Borel sets $A,B$, $\mu^{\otimes 2}(A \times B) \coloneqq \mu(A)\mu(B)$.
    \item $C^\infty_c$ is the space of test functions, i.e. infinitely differentiable functions with compact support.
    \item Convolution is denoted by $\star$.
    \item If $(M,g)$ is a Riemannian manifold, $d_M(x,y)$ denotes the geodesic distance between $x,y \in M$.
\end{itemize}
\section{Polyak-Lojasiewicz Inequality and Exponential Convergence}\label{SectionPL}

Our goal is to prove that a Wasserstein gradient flow curve $\mu_t$ of the energy $E_\nu$ converges to the target $\nu$.
In finite dimensions, a standard condition for convergence is the Polyak-Lojasiewicz inequality.

\begin{definition}[Polyak-Lojasiewicz condition]
    Let $f : \mathbb{R}^d \rightarrow \mathbb{R}$ be a differentiable function. It is said to satisfy the Polyak-Lojasiewicz condition (PL condition) with parameter $\lambda > 0$ if the following inequality holds for all $x \in \mathbb{R}^d$:
    \begin{equation*}
        \frac{1}{2}|\nabla f(x)|^2 \geq \lambda (f(x)-f^*)\,.
    \end{equation*}
\end{definition}

This condition is weaker than many other classical conditions, including strong convexity, weak strong convexity, or the restricted secant inequality, see \cite{PlCondKarimi} for a review. With this condition, an exponential convergence rate to the global minimum can be proven. Moreover, we can use a weaker dynamical version assumption. When examining the convergence of a gradient flow curve $x_t$ in the Euclidean space, defined by $\dot{x} = -\nabla f(x)$, it is sufficient for this inequality to hold along the curve. In the subsequent discussion, we focus on localized PL inequalities along specific curves, as opposed to global ones.

For gradient flows on measures, this type of inequality is also known under the name of \emph{entropy-entropy production} inequality in the context of gradient flow \cite{kondratyev2016new}. In general, such inequalities are functional inequalities. For instance, the log-Sobolev inequality is a PL inequality for the entropy under the Wasserstein geometry. Similarly, in \cite{kondratyev2016new}, the authors use such a generalized Beckner inequality to obtain PL.
To define an analog inequality in Wasserstein spaces we need the following chain rule \cite[Proposition 10.3.18]{ambrosio2005gradient}.

\begin{proposition}\label{ChainRule}
    Let $\mathcal{F}$ be a proper lower semicontinuous functional and $\mu_t$ be an absolutely continuous curve with tangent velocity $v_t$. We suppose $\mathcal{F} \circ \mu$ is approximately differentiable in time almost everywhere and that for all $t$ the set $\partial \mathcal{F}(\mu_t)$ of vector field subdifferentials $\xi \in L_2(\mu_t)$ (see Definition \ref{VectSubDiff}) is non-empty. Then for any $\xi_t \in \partial \mathcal{F}(\mu_t)$ we have:
    \begin{equation*}
        \frac{d}{dt} \mathcal{F}(\mu_t) = \int v_t \cdot \xi_t d\mu_t \,.
    \end{equation*}
\end{proposition}
For functionals that are regular enough, we have $v_t = -\nabla \frac{\delta \mathcal{F}}{\delta \mu}(\mu_t) $ and $\nabla \frac{\delta \mathcal{F}}{\delta \mu}(\mu_t) \in \partial \mathcal{F}(\mu_t)$ so that:
\begin{equation}\label{VarFCurve}
    \frac{d}{dt} \mathcal{F}(\mu_t) = -\left\|\nabla \frac{\delta \mathcal{F}}{\delta \mu}(\mu_t)\right\|_{L_2(\mu_t)}^2 \,.
\end{equation}
This motivates our definition of the Polyak-Lojasiewicz in Wasserstein spaces.
\begin{definition}
    Let $\mathcal{F}$ be a  functional as in Proposition \ref{ChainRule}. Suppose that its global minimum is equal to 0. Let $\mu_t$ be an absolutely continuous curve. The functional $\mathcal{F}$ is said to satisfy a Polyak-Lojasiewicz inequality with parameter $\lambda > 0$ along the curve $\mu_t$ if for all $t > 0$:
    \begin{equation}\label{def :PL ineq}
        \left\|\nabla \frac{\delta \mathcal{F}}{\delta \mu}(\mu_t)\right\|_{L_2(\mu_t)}^2 \geq \lambda \mathcal{F}(\mu_t)\,.
    \end{equation}
\end{definition}
If this inequality holds, we get, for regular enough functionals, thanks to equation \eqref{VarFCurve} and Gronwall's lemma:
\begin{equation*}
    \mathcal{F}(\mu_t) \leq \mathcal{F}(\mu_0) e^{-\lambda t} \,,
\end{equation*}
proving global exponential convergence.

\subsection{Regularity and Polyak-Lojasiewicz Inequality for the Energy Functional}

In this section, we prove a \emph{partial local} Polyak-Lojasiewicz inequality, in the sense that the constant $\lambda$ in \ref{def :PL ineq} does depend on $\mu$. It is notably weaker than a global PL inequality. The issue is that this inequality is only \emph{partial and local}, in the sense that it depends on a positive lower bound on $\mu$ that may not exist. In finite dimensions, obtaining a local Polyak-Lojasiewicz inequality (with a $\lambda$ that may depend on a given neighborhood) is still informative on the landscape of the energy functional. Specifically, it implies that there are no critical points other than global minima. However, in our case,  since this inequality doesn't hold for every $\mu$, a similar conclusion cannot be directly derived.

When the subdifferential of $E_\nu$ is non-empty, we can characterize it, using the same ideas as in \cite[Prop 4.3.1]{bonaschi}. The proof is essentially the same for both Energy Distance in the Euclidean space and Coulomb kernels both in the Euclidean space and on Riemannian manifolds, but crucial parts about avoiding the singularity rely on different arguments (see Appendix \ref{proof lemma 1}).
\begin{lemma}\label{MinimNormSubdiff}
    Let $\mu$ be a probability measure such that $\partial E_\nu (\mu)$ is non-empty. Then the vector field $\nabla \frac{\delta E_\nu}{\delta \mu}(\mu) = (\nabla G) \star (\mu - \nu)$ satisfies
    \begin{equation*}
        \left\|\nabla \frac{\delta E_\nu}{\delta \mu}(\mu)\right\|_{L_2(\mu)} \leq |\partial E_\nu |(\mu)\,.
    \end{equation*}
\end{lemma}

Then we need a regularity result about our functional, which is proven in Appendix \ref{proof lemma 1}.

\begin{lemma}\label{DiffRegMeas}
    Let $\mu, \nu \in \mathcal{P}_2^r$ be density measures regarding the Lebesgue measure or the volume measure on a manifold. Then the functional $E_\nu$ has a non-empty subdifferential, in the sense of Definition \ref{VecSubDiff}. Moreover, $\nabla \frac{\delta E_\nu}{\delta \mu}(\mu) \in \partial E_\nu (\mu)$.
\end{lemma}
Now, combining these two lemmas, the element of minimal norm of $\partial\mathcal{F}(\mu_t)$ is precisely the time-dependent vector field $v_t = \nabla \frac{\delta E_\nu}{\delta \mu}(\mu_t)$, which drives mass transfer along time for the curve $\mu_t$. In other words, the following property holds.

\begin{proposition}
    Let $\mu_t$ be a solution of equation \eqref{heatflow} in a weak sense. We suppose $\mu_t \in \mathcal{P}_2^r $ at all times. Then it is a gradient flow of the functional $E_\nu$ starting from $\mu_0$, in the sense of Definition \ref{PointGrad}.
\end{proposition}

Then we can formulate a condition that implies a \emph{local} Polyak-Lojasiewicz inequality in a compact manifold.

\begin{proposition}[Local Polyak-Lojasiewicz inequality]\label{PLIneq}
    Let $(M,g)$ be a closed Riemannian manifold and $\mu,\nu$ be two measures with density w.r.t. the volume measure on $M$, such that $\log(\mu)$ is bounded below.
    Then, it holds:
    \begin{equation}\label{EqPL}
      E_\nu(\mu)   \leq \frac{1}{\underline{\mu}}\left\|\nabla \frac{\delta E_\nu}{\delta \mu}(\mu)\right\|_{L_2(\mu)}^2 \,,
    \end{equation}
    where $\underline{\mu}$ is a lower bound for $\mu$ on $M$.
\end{proposition}

\begin{proof}
The proof is straightforward since inequality \ref{EqPL} is exactly:
\begin{equation*}
        \int_{M} | \nabla \varphi_{\mu - \nu}(x) |^2 \,d\!\operatorname{vol}(x) \leq \frac{1}{\underline{\mu}} \int_{M} | \nabla \varphi_{\mu - \nu}(x) |^2 \,d\mu(x) \,,
    \end{equation*}
    where $\operatorname{vol}$ is the volume measure on $M$. The inequality follows
    from $\underline{\mu} d\!\operatorname{vol} \leq \mu$.
\end{proof}

However, in a non-compact setting, a probability measure such that $\log(\mu)$ is bounded from below does not exist. The question of formulating a condition implying a Polyak-Lojasiewicz inequality in the Euclidean space therefore remains an open question.

A crucial point in proving global convergence of a gradient flow that satisfies a local and partial PL inequality is to ensure that the constant $\lambda(\mu(t))$ is well-defined and does not deteriorate too rapidly along the curve $\mu(t)$. In the next section, we show that $\lambda$ can be taken constant if the flow is sufficiently regular for all time.

\subsection{Exponential Convergence for Globally Regular Data}

First, we begin with a stability result, where the regularity of the density $\mu_t$ over time implies a global Polyak-Lojasiewicz inequality.

\begin{proposition}[Stability of the Polyak-Lojasiewicz condition]\label{StabPL}
Let $\mu_0,\nu$ be two $\mathcal{C}^1$ densities on $M$ such that $\log(\mu_0)$ and $\log(\nu)$ are bounded. Then if the associated equation $\eqref{heatflow}$ admits a continuous density, it satisfies
\begin{equation*}
    \min (\min \mu_0,\min \nu) \leq \mu_t(x) \leq \max (\max \mu_0,\max \nu)\,.
\end{equation*}
\end{proposition}

\begin{proof}[Sketch of proof]
    We use the regularity and an optimality argument.
    Let us define $\underline{x}(t) \in \argmin \, \mu_t(x)$ and $\overline{x}(t) \in \argmax \, \mu_t(x)$.
    Since the manifold is closed and $\mu_t$ is $\mathcal{C}^1$ we have: 
    \begin{equation*}
        \nabla \mu_t(\underline{x}(t)) = \nabla \mu_t(\overline{x}(t)) = 0 \,.
    \end{equation*}
    Moreover, using the regularity we obtain:
    \begin{align*}
        \partial_t \mu_t &= -\nabla \cdot (\mu_t v_t) \\
        &= -v_t \cdot \nabla \mu_t - \mu_t \nabla \cdot v_t \\
        \partial_t \mu_t &= -v_t \cdot \nabla \mu_t - \mu_t (\mu_t-\nu)\,.
    \end{align*}
    Instantiating the previous inequality at $\underline{x}(t)$ for example we get:
    \begin{equation*}
        \partial_t \mu_t(\underline{x}(t)) = (\min \mu_t )\nu(\underline{x}(t)) - (\min \mu_t)^2 \,.
    \end{equation*}
    After a brief study of the phase diagram (see the proof of Lemma \ref{ContrDensFt}) we obtain our result.
    A similar argument applies to $\overline{x}(t)$.
\end{proof}

This stability of the PL condition allows us to formulate a convergence result, if we make strong assumptions on the global existence and the regularity of $\mu(t,x)$.
\begin{proposition}[Exponential convergence]\label{ThExpConvSuffic}
Let $\mu,\nu$ be two measures with continuous densities and bounded logarithms on a closed Riemannian manifold $M$.
If the density of the curve $\mu_t$ generated by the gradient flow of $E_\nu$ is $C^1_{t,x}$ in both time and space for all $t \in \mathbb R_{\geq 0}$ and $x \in M$, then exponential convergence of $\mu_t$ to $\nu$ holds in the following two ways:
\begin{equation*}
    \begin{cases}
        E_\nu(\mu_t) \leq  E_\nu(\mu_0)e^{-\lambda t} \\
        W_2^2(\mu_t,\nu) \leq \frac{4}{\lambda}E_\nu(\mu_0)e^{-\lambda t}\,,
    \end{cases}
\end{equation*}
where $\lambda = \min(\underline{\mu_0},\nu)$.
\end{proposition}
\begin{proof}
    The proof is a straightforward application of the Polyak-Lojasiewicz inequality.
\end{proof}
In Theorem \ref{ThGlobalConvergenceFinal}, we are able to relax the assumption on the regularity of the solution to  H\"older continuous solutions. 
Therefore, we have just reduced the global convergence problem to a global existence and regularity problem. In the next section, we prove the global existence of H\"older continuous solutions if the initial and target densities are H\"older continuous.
Our convergence result is then stated in Theorem \ref{ThGlobalConvergenceFinal}.

\subsection{Well-Posedness in Hölder Spaces}

In this section, we consider solutions of the PDE:
\begin{equation}\label{EqHold}
\begin{cases}
        \partial_t \mu_t + \nabla \cdot (\mu_t v_t) = 0\\
        v_t \coloneqq -\nabla G \star (\mu_t - \nu)\,.
        \end{cases}
\end{equation}
The velocity vector field $v_t$ satisfies 
$
    \nabla \cdot v_t = \mu_t - \nu \,.
$
Let us consider a solution $\mu_t$ of equation \eqref{EqHold} on an open interval $]0, T[$. According to Definition \ref{ContEq}, it is an absolutely continuous curve associated with the time-dependent vector field $v_t\coloneqq - \grad \varphi$, where $\Delta \varphi_t = \mu_t - \nu$. If $\mu_0$ has a density, $\mu_t$ also has a density at least for short times.

\subsubsection{On the Euclidean Space}

This subsection is devoted to the proof of the following theorem:
\begin{theorem}\label{GlobExistHolder}
    Let $\mu_0$ and $\nu$ be Hölder continuous densities with compact support. Then, equation \eqref{heatflow} admits a unique global solution $\mu_t$ that is Hölder continuous at all times.
\end{theorem}
We use techniques from \cite{bertozzi2012} and adapt them to our case. The standard approach involves rewriting the problem in Lagrangian coordinates, following the particle flow, and leveraging ODE results to control the norm of the density.

Let us set the notations. We denote $G$ as the Coulomb kernel, which is repulsive and a solution of $\Delta G = -\delta$. We define the particle flow $\psi_t$ by $\psi_0 = \operatorname{Id}$ and 
\begin{equation} \label{FlowPart}
    \frac{d}{dt} \psi_t = v_t \circ \psi_t \,.
\end{equation}
Let us consider the evolution of the time-dependent function $f_t \coloneqq \mu_t \circ \psi_t$. As $\mu_t = f_t \circ \psi_t^{-1}$, the existence of $\mu_t$ is linked to the existence of $f_t$ and $\psi_t$. This dependence will be made precise in Lemma \ref{HoldNormMutVraieRes}.
We can rewrite the particle flow equation \eqref{FlowPart}, for $\alpha \in \mathbb{R}^d$:
\begin{align*}
    \frac{d}{dt}\psi_t(\alpha) &= v_t \circ \psi_t (\alpha)
    = -\int \nabla G (\psi_t(\alpha) - y)(\mu_t - \nu)(y)dy \\
    \frac{d}{dt}\psi_t(\alpha) &= -\int \nabla G (\psi_t(\alpha) - \psi_t(\alpha'))\det (d \psi_t(\alpha'))\mu_t(\psi_t(\alpha'))d\alpha' +\int \nabla G(\psi_t(\alpha)-y)\nu(y)dy \,.
\end{align*}
Let us note $J_t(\alpha) := \det (d \psi_t(\alpha))$ and remark: 
\begin{equation*}
    \frac{d}{dt} J_t(\alpha) = J_t(\alpha)(\mu_t - \nu)\circ \psi_t(\alpha) \,,
\end{equation*}
so that:
\begin{equation}\label{ChgtVar}
    \frac{d}{dt}\left[t \mapsto J_t(\alpha) \mu_t(\psi_t(\alpha)) \right] = 0\,,
\end{equation}
and as $J_0(\alpha) = 1$ we get
$ J_t(\alpha)\mu_t(\psi_t(\alpha)) = \mu_0(\alpha) \,.$
We end up with the Lagrangian formulation of the flow 
\begin{equation}\label{LagFlow}
    \frac{d}{dt} \psi_t(\alpha) = -\int \nabla G(\psi_t(\alpha)-\psi_t(\alpha'))\mu_0(\alpha)d\alpha +\int \nabla G(\psi_t(\alpha)-y)\nu(y)dy = F(\psi_t(\alpha))\,. 
\end{equation}
We solve this equation in the Banach space $\mathcal{B}$ defined by
$\mathcal{B} \coloneqq \left\{ \psi : \mathbb{R}^d \rightarrow \mathbb{R}^d \, |\, \|\psi\|_{1,\gamma} < \infty \right\},$
where $\|\psi\|_{1,\gamma} := |\psi(0)| + \|d \psi \|_\infty + |d\psi |_\gamma$ with $| \cdot|_\gamma$ the Hölder semi-norm.
Now we can state our first local existence result.
\begin{proposition}\label{LocExistHold}
    Let $\mu_0$ and $\nu$ be density measures in $\mathcal{B}$ with compact support. Then equation \eqref{LagFlow} with initial condition $\psi_0 = Id$ admits a unique solution on a maximal time interval $[0, T[$. Either $T$ is infinite or the Banach norm $\|\psi_t\|_{1,\gamma}$ blows up as $t \rightarrow T$.
\end{proposition}

\begin{proof}
    As $F$ is Lipschitz \cite{bertozzi2012} and \cite[Chap 4]{majda_bertozzi_2001}) the Picard theorem in Banach spaces proves the result.
\end{proof}
Now let us prove Theorem \ref{GlobExistHolder}, that is $T = +\infty$ in the preceding proposition. To show this we suppose $T< \infty$ in the whole following discussion and show that $\|\psi_t\|_{1,\gamma}$ is bounded uniformly in time on the interval $[0, T[$. In \cite{bertozzi2012}, an explicit formula is found for $f_t$, which cannot be done in our case. However, we can control the evolution of $f_t$.

\begin{lemma}\label{ContrDensFt}
    The quantity $\|f_t\|_\infty$ is uniformly bounded on $[0,T[$.
\end{lemma}

\begin{proof}
    We can write:
    \begin{equation*}
        \nabla \cdot (v_t \mu_t) = \mu_t \nabla \cdot v_t + v_t \cdot \nabla \mu_t = \mu_t^2 - \mu_t \nu + v_t \cdot \nabla \mu_t \,,
    \end{equation*}
    so that we get:
    \begin{align}\label{LogistEqFt}
        \frac{d}{dt} f_t(\alpha) &= \partial_t \mu_t (\psi_t(\alpha)) + \nabla \mu_t(\psi_t(\alpha)) \cdot v_t(\psi_t(\alpha)) \nonumber \\
        &= -\nabla \cdot (v_t\mu_t)(\psi_t(\alpha)) + \nabla \mu_t(\psi_t(\alpha)) \cdot v_t(\psi_t(\alpha)) \\
        \frac{d}{dt} f_t(\alpha) &= f_t(\alpha)\nu(\psi_t (\alpha)) - f_t^2(\alpha) \nonumber \,.
    \end{align}
    Although $\nu \circ \psi_t$ is of course not constant, this equation, which resembles a logistic equation, is well-behaved. We can study its phase diagram. First, if $f_0$ is a positive function then $f_t$ will be positive too. Moreover, if $f_t(\alpha)>\|\nu\|_\infty$ then $\frac{d}{dt}f_t(\alpha)$ is negative, so $t \mapsto f_t(\alpha) $ is locally decreasing. This shows that for all $\alpha \in \mathbb{R}^d$:
    \begin{equation*}
        \min (\min f_0 , \min \nu) \leq f_t(\alpha) \leq \max (\|f_0\|_\infty , \|\nu\|_\infty) \,.
    \end{equation*}
\end{proof}
This control allows us to bound the Holder norm by a quantity that depends on $\psi_t$.

\begin{lemma}\label{HoldNormMutVraieRes}
    Let $\mu_t$ be a solution defined as in Proposition \ref{LocExistHold}, where $\nu \in L_\infty$ satisfies 
    $| \nu |_\gamma < \infty$. Then:
    \begin{equation*}
        |\mu_t|_\gamma \leq C \|d \psi_{t}\|^\gamma_\infty \left(\int_0^t (1 + \|d \psi_{s}^{-1}\|_\infty^\gamma)ds \right)\,,
    \end{equation*} for some positive constant $C>0$.
\end{lemma}

\begin{proof}
    As $\mu_t = f_t \circ \psi_t^{-1}$, using that $|f \circ g|_\gamma \leq |f|_\gamma \|dg\|_\infty^\gamma$ for general functions $f,g$, we get the first estimate:
    \begin{equation*}
        |\mu_t|_\gamma \leq |f_t|_\gamma \|d\psi_t^{-1}\|_\infty^\gamma \,.
    \end{equation*}
    Now, to control $|f_t|_\gamma$, we can take the Holder semi norm of equation \eqref{LogistEqFt}, using the fact that $\nu$ is bounded and Holder continuous, that $\|f_t\|_\infty$ is bounded, and the property $|fg|_\gamma \leq |f|_\gamma \|g\|_\infty + \|f\|_\infty |g|_\gamma$ to get:
    \begin{align*}
        \frac{d}{dt}|f_t|_\gamma &\leq |f_t \cdot \nu\circ \psi_t| + |f_t^2|_\gamma \\
        &\leq |f_t|_\gamma \|\nu \circ \psi_t\|_\infty + \|f_t\|_\infty |\nu \circ \psi_t|_\gamma + 2\|f_t\|_\infty |f_t|_\gamma \\
        &\leq (\|\nu\|_\infty + 2\|f_t\|_\infty) |f_t|_\gamma + \|f_t\|_\infty |\nu|_\gamma \|d\psi_t \|_\infty^{\gamma} \\
        \frac{d}{dt}|f_t|_\gamma &\leq C_1 |f_t|_\gamma + C_2\|d\psi_t \|_\infty^{\gamma}\,,
    \end{align*}
    for some positive constants $C_1,C_2$. 
    Then we apply Gronwall's lemma with time-dependent terms, stating that if $y$ is differentiable and $a,b$ are continuous functions such that $\dot{y} \leq ay + b$, then $y(t) \leq y(0)e^{\int_0^t a(s)ds}+\int_0^t b(s)e^{\int_0^t a(u)du - \int_0^s a(u)du}ds$.
    We obtain, as $t<T<\infty$:
    \begin{equation*}
        |f_t|_\gamma \leq |f_0|_\gamma \exp(C_1 t) + \int_0^t C_2\|d\psi_s\|_\infty^\gamma \exp (C_1(t-s))ds \leq C \left( 1 + \int_0^t\|d\psi_s\|_\infty^\gamma ds \right)\,,
    \end{equation*}
    ending the proof.
\end{proof}
Differentiating the particle equation \eqref{FlowPart} and taking the $L_\infty$ norm we get:
\begin{equation*}
    \frac{d}{dt}\|d\psi_t \|_\infty \leq \|d v_t\|_\infty \|d \psi_t\|_\infty\,.
\end{equation*} 
This gives by Gronwall's lemma:
\begin{equation}\label{EstimGradFlotVit}
    \|d\psi_t \|_\infty \leq C\exp{\int_0^t \|d v_s\|_\infty ds}\,.
\end{equation}
Let us observe that $\psi_t^{-1}$ satisfies a similar bound.
\begin{lemma}
    Let $\psi_t^{-1}$ be defined as the inverse flow of $\psi_t$. Then the majoration above is true for $\psi_t^{-1}$, i.e. for some constant $C>0$:
\begin{equation}\label{EstimGradFlotInvVit}
    \|d\psi_{t}^{-1} \|_\infty \leq C\exp{\int_0^t \|d v_s\|_\infty ds}\,.
\end{equation}
\end{lemma}

\begin{proof}
    Let $\alpha := \psi_s(y)$ for some $y \in \mathbb{R}^d$ and $s>0$. Then if $0 \leq t<s$, we have $\alpha \in \psi_t(\mathbb{R}^d)$, as $\alpha = \psi_t \circ \psi_{s-t}(y)$ by semi-group property. We get $\psi_t^{-1}(\alpha) = \psi_{s-t}(y)= \psi_{s-t} \circ \psi_s^{-1} (\alpha) $. Differentiating both sides of the equation in time we get:
    \begin{equation*}
        \partial_t \psi_t^{-1} (\alpha) = \partial_t \psi_{s-t}\circ \psi_s^{-1} (\alpha) = -v_{s-t}\circ \psi_{s-t} \circ \psi_s^{-1}( \alpha) = -v_{s-t}\circ \psi_t^{-1}( \alpha)\,.
    \end{equation*}
    This leads to the inequality:
    \begin{equation*}
        \frac{d}{dt} \|d\psi_t^{-1}\|_\infty \leq \|d v_{s-t}\|_\infty \|d \psi_t^{-1}\|_\infty\,, 
    \end{equation*}
    so that by Gr\"onwall lemma:
    \begin{equation*}
        \|d\psi_t^{-1}\|_\infty \leq C\exp \left(\int_0^t \|d v_{s-u}\|_\infty du \right) \,.
    \end{equation*}
    We can choose $s=t$ in the previous inequality to get:
    \begin{equation}\label{InegInvFlo}
        \|d\psi_t^{-1}\|_\infty \leq C\exp \left(\int_0^t \|d v_{u}\|_\infty du \right) \,.
    \end{equation}
\end{proof}
With this result, we managed to bound from above all of our quantities by functions of $dv_t$.\\
 \noindent
    \textbf{On the derivative of our velocity field.}
    The following arguments are presented in \cite[Sections 2.4.2,3]{majda_bertozzi_2001}.
    The kernel $G_2 \coloneqq d_x \grad G$ is homogeneous of degree -N. Due to this, the singularity at the diagonal cannot be integrated. However, it has mean-value zero and defines a singular integral operator through the convolution:
    \begin{equation*}
        G_2 \star f(x) = PV \int G_2(x,y)f(y)dy \coloneqq \underset{\varepsilon \rightarrow 0}{\lim} \int_{d(x,y)>\varepsilon} G_2(x,y)f(y)dy \,.
    \end{equation*}
    Using the same arguments than in \cite[Prop 2.20]{majda_bertozzi_2001}, we know that for a velocity field defined by $\grad G \star f$ with $G$ the Coulomb kernel and $f \in \mathcal{C}^\gamma(\mathbb{R}^d;\mathbb{R}^d)$ we have:
    \begin{equation*}
        dv_t(x) = PV \int G_2(x,y)f(y)dy \,.
    \end{equation*}
    The following result, found in \cite[Lemma 4.5 and 4.6]{majda_bertozzi_2001} and \cite[Lemma 2.2]{bertozzi2012}, is expressed in terms of principal value integral in \cite{bertozzi2012} but used in the following form.

    \begin{lemma}\label{PotThEstim}
    Let $f \in \mathcal{C}^\gamma(\mathbb{R}^d;\mathbb{R}^d)$ be a compactly supported function in a ball of radius $R$. We define $v \coloneqq \nabla G \star f$. Then, for some positive constant $C$ independent of $f$ and $R$ we have:
        \begin{align*}
             \| dv \|_\infty  &\leq C \left[  |f|_\gamma \varepsilon^\gamma + \max\left(1; \log\left( \frac{R}{\varepsilon}\right)\right)\|f\|_\infty\right] ; \forall \varepsilon >0 \,,  \\
             |dv|_\gamma &\leq C |f|_\gamma  \,. 
        \end{align*}
\end{lemma}
To use it we need to confine the support of $\mu_t$.
First, we can bound the velocity field $v_t$.
\begin{lemma}\label{BoundVitEucl}
    Let $\mu_t$ be a solution as defined in Proposition \ref{LocExistHold} and $v_t$ the associated velocity field. Then for some constant $C>0$ only depending on the dimension:
    \begin{equation*}
        \|v_t\|_\infty \leq C ( \|\mu_t - \nu \|_\infty + \|\mu_t - \nu \|_1) \,.
    \end{equation*}
\end{lemma}

\begin{proof}
    We write for $x \in \mathbb{R}^d$, knowing $|\nabla G(x)|$ is proportional to $|x|^{d-1}$:
    \begin{align*}
        |v_t(x)| &\leq \left[\int_{B(0,1)} + \int_{B(0,1)^c}\right] |\nabla G(x-y)(\mu_t-\nu)(y)|dy \\
        &\leq \|\mu_t-\nu\|_\infty \int_{B(0,1)}|\nabla G(x-y)|dy + \int_{B(0,1)^c} |\mu_t - \nu|(y)dy \\
        |v_t(x)| &= C (\|\mu_t-\nu\|_\infty + \|\mu_t-\nu\|_1)\,.
    \end{align*}
\end{proof}
This allows us to bound the growth of the support.
\begin{lemma}\label{Confinement}
    Let $\mu_t$ be a solution defined as in Proposition \ref{LocExistHold}. Suppose that the support of $\mu_0$ is contained in the ball of center $0$ and radius $R_0>0$. Then there exists a positive constant $C>0$ such that the support of $\mu_t$ is contained in $R(t)\coloneqq R_0 + Ct$.
\end{lemma}

\begin{proof}
    We use the inequality from Lemma \ref{BoundVitEucl}.
    As $\|\mu_t\|_\infty = \|f_t\|_\infty \leq C$ by Lemma \ref{ContrDensFt}, the first term is uniformly bounded in $[0,T[$. As $\mu_t$ and $\nu$ are probability densities, so is the second term. This proves the result.
\end{proof}

Having proved that $\mu_t$ has bounded support, we can apply the first estimate from Lemma \ref{PotThEstim} with $\varepsilon = \left[ \|\mu_t - \nu\|_\infty / |\mu_t - \nu|_\gamma \right]^{1/\gamma}$ to establish the existence of $C$ independent of $t$ such that:
\begin{equation*}
    \|dv_t\|_\infty \leq C \left[ \|\mu_t - \nu\|_\infty + \max \left(1, \log \left(\frac{R(t)|\mu_t - \nu|^{1/\gamma}_\gamma}{\|\mu_t - \nu\|^{1/\gamma}_\infty} \right)\|\mu_t - \nu\|_\infty \right)  \right].
\end{equation*}
Thanks to Lemma \ref{ContrDensFt}, $\|\mu_t - \nu\|_\infty$ is bounded and by Lemma \ref{Confinement}, as $T<\infty$ we get the existence of $C_1,C_2 >0$ such that:
\begin{equation}\label{GradVitNormHoldMu}
    \|dv_t\|_\infty \leq C_1 + C_2 \log(|\mu_t - \nu|_\gamma).
\end{equation}
Now, we are ready to prove our first real boundedness result.
\begin{proposition}
    The quantity $\|dv_t \|_\infty$ is uniformly bounded on the time interval $[0,T[$.
\end{proposition}

\begin{proof}
Injecting the inequality from Lemma \ref{HoldNormMutVraieRes} into equation \eqref{GradVitNormHoldMu}, we get the existence of constants such that:
\begin{equation*}
    \|dv_t \|_\infty \leq A + B \log(\|d \psi_t\|_\infty) + C \log \left(\int_0^t (1 + \|d \psi_{s}^{-1}\|_\infty^\gamma)ds \right) \,.
\end{equation*}
We write, using inequality \eqref{EstimGradFlotInvVit}, the fact that $s\leq t$ in the integrals and that $\|dv_t\|_\infty$ is a positive function:
\begin{align*}
    \log \left(\int_0^t  \|d \psi_{s}^{-1}\|_\infty^\gamma ds \right) &\leq C\log \left(\int_0^t \gamma \exp \left( \int_0^s \|dv_u \|_\infty du  \right)ds \right) \\
    &\leq C\log \left(\int_0^t \gamma \exp \left( \int_0^t \|dv_u \|_\infty du  \right)ds \right) \\
    &\leq C \log \left( \gamma t \exp \left( \int_0^t \|dv_u \|_\infty du  \right)\right) \\
    \log \left(\int_0^t  \|d \psi_{s}^{-1}\|_\infty^\gamma ds \right) &\leq C \left( \log t + \int_0^t \|dv_u \|_\infty du \right) \,,
\end{align*}
and we obtain the final differential inequality:
\begin{equation*}
    \|dv_t \|_\infty \leq C \left( 1 + \int_0^t \|dv_s \|_\infty ds +\log(t)  \right) \,.
\end{equation*}
Once again, Gronwall's lemma applies to $t \mapsto \|dv_t \|_\infty$ and we get the existence of constants $C_1,C_2$ such that
$
    \|dv_t \|_\infty \leq C_1 \exp(C_2t)\,,
$
showing the result as $t<T<+\infty$.
\end{proof}
This immediately implies:
\begin{proposition}
    The quantities $\|d\psi_t\|_\infty$, $\|d\psi_t^{-1}\|_\infty$ and $|\mu_t|_\gamma$ are uniformly bounded in time on $[0,T[$.
\end{proposition}

\begin{proof}
    Using the fact that $\|dv_t \|_\infty$ is uniformly bounded in $[0,T[$, inequalities \eqref{EstimGradFlotVit}, \eqref{EstimGradFlotInvVit} and Lemma \ref{HoldNormMutVraieRes} show the result.
\end{proof}
Finally, we can control our last term:
\begin{proposition}
    The quantity $|d\psi_t|_\gamma$ is uniformly bounded in $[0,T[$.
\end{proposition}

\begin{proof}
Differentiating the particle equation \eqref{FlowPart}, taking the Hölder $|\cdot|_\gamma$ semi-norm and applying the preceding proposition along with the second potential theory estimate from Lemma \ref{PotThEstim} we get: 

\begin{align}\label{EstimGradnHoldXt}
    \frac{d}{dt} |d\psi_t|_\gamma &\leq |d (v_t\circ \psi_t)|_\gamma \|d\psi_t\|_\infty + \|d (v_t\circ \psi_t)\|_\infty |d\psi_t|_\gamma \nonumber \\
    & \leq |dv_t|_\gamma \|d\psi_t\|^{1+\gamma}_\infty + \|d v_t\|_\infty  \|d\psi_t\|_\infty |d\psi_t|_\gamma \nonumber \\
    & \leq C_1|\mu_t|_\gamma \exp{\left((1+\gamma) \int_0^t \|d v_s\|_\infty ds \right)} + C_2 |d\psi_t|_\gamma \\
    \frac{d}{dt} |d\psi_t|_\gamma &\leq C_1 + C_2|d\psi_t|_\gamma \nonumber \,.
\end{align}

One last application of Gronwall's lemma ends the proof.

\end{proof}
We just proved that if $T< \infty$ and a solution of problem \eqref{LagFlow} in the Banach space $\mathcal{B}$ exists in $[0,T[$, then the Banach norm $\|\psi_t\|_{1,\gamma}$ is uniformly bounded on $[0,T[$. If the maximal time of existence $T$ satisfies $T<\infty$, we get a contradiction with the existence result of Proposition \ref{LocExistHold}, which states that if $T<\infty$ a finite time blowup of $\|\psi_t\|_{1,\gamma}$ occurs. This proves Theorem \ref{GlobExistHolder}.

\subsubsection{On a Compact Riemannian Manifold}

This Lagrangian formulation allows us to directly extend our result to a complete closed Riemannian manifold $(M,g)$. Indeed, taking $G$ the Coulomb kernel on $(M,g)$, the Lagrangian formulation \eqref{LagFlow} still holds. As Hölder regularity is a local property, the functional $F$ is still locally Lipschitz on the Banach space:
\begin{equation*}
    \mathcal{B}_M := \left\{ \psi : M \rightarrow TM \, |\, \|\psi\|_{1,\gamma} < \infty \right\} \,,
\end{equation*}
where $|\psi(0)| + \|d \psi \|_\infty + |d\psi |_\gamma$ with $\|\cdot\|_{\mathcal{C}_M^{0,\gamma}}$ the H\"older semi-norm on the manifold $M$, defined by:
    \begin{equation*}
        \|f\|_{\mathcal{C}_M^{0,\gamma}} \coloneqq \underset{x,y \in M}{\sup}\, \frac{|f(x)-f(y)|}{d_M(x,y)^\gamma} \,.
    \end{equation*}
    for scalar functions, and through parallel transport for tensors.

We get the existence of a flow $\psi_t$, at least locally. The rest of the proof is similar, the main difference being the equivalent of Lemma \ref{PotThEstim} on closed manifolds. 

    \begin{lemma}\label{PotThEstimMani}
    Let $u \in \mathcal{C}^\gamma(M)$ on a closed manifold $M$.
    Consider the equation:
    \begin{equation*}
        \Delta_M \varphi = u \,.
    \end{equation*}
    where $u$ is H\"older continuous. Then, for some positive constant $A,C$ (independent of $u$) and for all $\varepsilon >0$ we have the Schauder estimates
        \begin{align*}
        \| \varphi \|_{2,\infty} &\leq C \left( \| u \|_{\mathcal{C}_M^{0,\gamma}} \varepsilon^\gamma + \log\left(\frac{A}{\varepsilon}\right)\|u\|_\infty \right) \,, \\
        \| \varphi \|_{\mathcal{C}_M^{2,\gamma}} &\leq C \left( \|u\|_{\mathcal{C}_M^{0,\gamma}} + \|\varphi\|_{\infty} \right) \,.
        \end{align*}

    \end{lemma}
    \begin{proof}[Proof of the first inequality]
        The proof of the first inequality is almost the same as in the Euclidean space. We denote $G_2$ as the differential (in coordinates) of the gradient of $G$, i.e. $G_2 \coloneqq d_x \grad G$. Its singularity on the diagonal behaves like the Coulomb kernel, implying \cite[Theorem 4.13.c]{aubun1998}:
        \begin{equation*}
            G_2(x,y) = O(1/d_M(x,y)^{d})\,.
        \end{equation*}
        Moreover, in Riemannian manifolds:
        \begin{equation*}
            \int_{d_M(x,y)>\varepsilon} d_M(x,y)^{-d}dy = O(-\log \varepsilon) \, ,
        \end{equation*}
        and if $\alpha <d$:
        \begin{equation*}
            \int_{d_M(x,y)<\varepsilon} d_M(x,y)^{-\alpha}dy = O(\varepsilon^{d-\alpha}  )\, .
        \end{equation*}
        We write:
        \begin{equation*}
            \| \varphi \|_{2,\infty} = \left( PV \int_{d(x,y)\leq\varepsilon} + \int_{d(x,y)>\varepsilon} \right) G_2(x,y)f(y)dy = I_1(x)+I_2(x) \,.
        \end{equation*}
        As $\int_B G_2(x,y)dy = 0$ on metric balls, we get:
        \begin{align*}
            |I_1(x)| &= \left|\int_{d(x,y)\leq\varepsilon} G_2(x,y)(f(y)-f(x))dy \right| \\
            &\leq \int_{d(x,y)\leq\varepsilon} \|G_2(x,y)\| |f\|_{\mathcal{C}_M^{0,\gamma}} d(x,y)^\gamma dy \\
            &\leq C |f\|_{\mathcal{C}_M^{0,\gamma}} \int_{d(x,y)\leq\varepsilon} d_M(x,y)^{-d+\gamma}dy \\
            |I_1(x)| &\leq C |f\|_{\mathcal{C}_M^{0,\gamma}} \varepsilon^\gamma\,.
        \end{align*}
        For the second term:
        \begin{align*}
            |I_2(x)| &= \int_{d(x,y)>\varepsilon} G_2(x,y)f(y)dy \\
            &\leq C \|f\|_\infty \int_{d(x,y)>\varepsilon} d(x,y)^{-d}dy \\
            |I_2(x)| &\leq C \|f\|_\infty \log\left(\frac{A}{\varepsilon}\right) \,. 
        \end{align*}
        This proves the first estimate.
            \end{proof}
        \begin{proof}[Sketch of proof for the second inequality]
        The second estimate is proven, for a manifold embedded in the Euclidean space, with standard Schauder theory in $\mathbb{R}^d$ \cite[3.61]{aubun1998}. We use local normal coordinates for the Laplace-Beltrami operator. The bound in the Euclidean space can be applied thanks to the following metric control of the H\"older norm. We refer to the lemma below.
        \end{proof}
    \begin{lemma}
        There exists $r>0$ and $C>0$ such that, if $\Omega \subset M$ has diameter less than $r$, then if $u \in \mathcal{C}^{2,\gamma}$:
        \begin{equation*}
            \frac{1}{C} \|u\|_{\mathcal{C}_{Eucl}^{k,\gamma,\Omega}} \leq \|u\|_{\mathcal{C}_M^{k,\gamma}} \leq C \|u\|_{\mathcal{C}_{Eucl}^{k,\gamma,\Omega}}\,.
        \end{equation*}
        In this inequality, $\|u\|_{\mathcal{C}^{k,\gamma,\Omega}_{Eucl}}$ denotes the Euclidean H\"older norm in $\Omega$ when considered as a subset of $\mathbb{R}^d$.
        
    \end{lemma}

Now, we use these estimates to conclude the proof of global existence. In our case $\| \psi_t \|_{2,\infty} = \| dv_t \|_{0,\infty}$ and $\| \psi_t \|_{\mathcal{C}_M^{2,\gamma}} = \| dv_t \|_{\mathcal{C}_M^{0,\gamma}} $.
    Note that in the second inequality, since the manifold $M$ is closed we have $\|\psi\|_\infty \leq C$ for some positive constant since $\nabla \psi$ is bounded in $L^\infty$.
    We can use the first estimate from Lemma \ref{PotThEstimMani} with, once again, $\varepsilon = \left[ \|\mu_t - \nu\|_\infty / |\mu_t - \nu|_\gamma \right]^{1/\gamma}$ to obtain formula \eqref{GradVitNormHoldMu} on the manifold $M$. This allows to bound $\|dv_t\|_\infty$ in the same way as in the Euclidean space. 
    To bound  $|dv_t|_\gamma$, we use the second inequality of the lemma. The principle of the proof is the same as in the Euclidean space and equation \eqref{EstimGradnHoldXt} becomes:
    \begin{align}
    \label{EstimGradnHoldXtManifold} 
    \frac{d}{dt} |d\psi_t|_\gamma &\leq |d (v_t\circ \psi_t)|_\gamma \|d\psi_t\|_\infty + \|d (v_t\circ \psi_t)\|_\infty |d\psi_t|_\gamma \nonumber\\
    & \leq |dv_t|_\gamma \|d\psi_t\|^{1+\gamma}_\infty + \|d v_t\|_\infty  \|d\psi_t\|_\infty |d\psi_t|_\gamma \nonumber\\
    & \leq C_1\left(|\mu_t|_\gamma + \|v_t\|_\infty \right) \exp{\left((1+\gamma) \int_0^t \|d v_s\|_\infty ds \right)} + C_2 |d\psi_t|_\gamma \\
    \frac{d}{dt} |d\psi_t|_\gamma &\leq C_1 + C_2|d\psi_t|_\gamma \nonumber \,.
\end{align}
The rest of the proof follows the same steps as in the Euclidean space.
This proves:
\begin{theorem}[Global convergence for Hölder initial and target data]\label{ThGlobalConvergenceFinal}
    Let $\mu_0$ and $\nu_0$ be Hölder continuous probability densities on a closed manifold $(M,g)$. Consider the curve $\mu_t$ global solution of equation \eqref{EqHold}. We have just proved that it is defined and Hölder continuous at all times.
    Then 
    \begin{equation*}
    \begin{cases}
        E_\nu(\mu_t) \leq  E_\nu(\mu_0)e^{-\lambda t} \\
        W_2^2(\mu_t,\nu) \leq \frac{4}{\lambda}E_\nu(\mu_0)e^{-\lambda t}\,.
    \end{cases}
\end{equation*}
\end{theorem}

\begin{proof}
    Since $\log(\mu_t)$ is globally bounded from below thanks to Lemma \ref{ContrDensFt}, it satisfies a global Polyak-Lojasiewisz at all times. The rest of the proof is a straightforward application of this inequality.
\end{proof}

\section{Critical Points for the Wasserstein Flow}
\label{SecCriticalPoints}

In the previous section, we studied regular solutions of Wasserstein gradient flows, with smooth initial 
 and target data. Our proofs heavily relied on regularity results obtained through potential theory estimates.
In this section, we consider some given probability measures $\mu$ and $\nu$ that are, depending on the context, of finite energy for the Coulomb kernel or the Energy Distance kernel. Our goal is to study critical points for the Wasserstein flow of the MMD energy $E_\nu$. 

\subsection{Critical Points and Lagrangian Critical Points for MMD Wasserstein Gradient Flows}

For an arbitrary function $\mathcal{F}$, we define critical points of the associated Wasserstein gradient flows. Intuitively, they correspond to measures where the discrete JKO steps are blocked, in direct analogy with gradient flows in finite dimensions.

\begin{definition}[Wasserstein critical point]\label{WCritPoints}
\label{def: crit point}
    Let $\mu$ be a probability measure such that $\mathcal{F}(\mu) < +\infty$. We say that $\mu$ is a Wasserstein critical point of $\mathcal{F}$ if there exists $\tau_0 >0$ such that for all $0<\tau \leq \tau_0$ we have:
    \begin{equation*}
        \mu \in \underset{\rho \in \mathcal{P}(\mathbb{R}^d)}{\arg \min}\, \mathcal{F}(\rho) + \frac{1}{2\tau}W_2^2(\mu,\rho)\,.
    \end{equation*}
\end{definition}

We study a sub-class of critical points, which we call Lagrangian critical points or displacement critical points.

\begin{definition}
    A probability measure $\mu$ is said to be a Lagrangian critical point for a functional $\mathcal{F}$ if, on $\text{supp}(\mu)$ we have
    \begin{equation*}
        \nabla \frac{\delta \mathcal{F}}{\delta \mu}(\mu) = 0\,.
    \end{equation*}
\end{definition}

For differentiable $\lambda$-convex functionals \cite{santambrogio2015optimal}, the two definitions are equivalent. In general settings, we cannot deduce one from the other, as the quantity above may not even belong to $\partial \mathcal{F}(\mu)$. However, in our cases, we can prove a partial result. Indeed, using the general differentiation result on JKO steps given by proposition \ref{DiffJKO} we see that the following Proposition is true.

\begin{proposition}
    Let $\mu$ be a critical point as defined in \ref{def: crit point}. Then $0 \in \partial \mathcal{F}(\mu)$ (where $0$ is seen as an element of $L_2(\mu)$).
\end{proposition}

Combining this result and lemma \ref{MinimNormSubdiff}, we see that if $\mu$ is a critical point of the $G$-energy $E_\nu$, then the gradients of the potentials $G \star \mu$ and $G \star \nu$ are equal $\mu$-almost everywhere. This observation is the key argument enabling us to prove the results of this section.

\subsection{Characterization of Lagrangian Critical points for MMD Wasserstein Gradient Flows}

The main result of this section is presented in the following theorem. 

\begin{theorem}\label{eq measure open}
    Let $\mu$ be a Lagrangian critical point for the MMD functional $E_\nu^G$. Then
    \begin{equation*}
        \mu_{|\text{Int(supp}(\mu))} = \nu_{|\text{Int(supp}(\mu))}\,,
    \end{equation*}
    holds  if G is the Coulomb kernel  
        or if G is the Energy Distance kernel and $d$ is odd. 
\end{theorem}

This formulation may appear surprising. We obtain results only in the interior of the support of our measure, even though it may be empty, or $\mu$ may be composite, being the sum of a density measure and a singular measure for example. As the proof for the Energy Distance kernel is more involved, we focus on the first statement for now.

In both cases, the first variation of the $E_\nu$ is given by: 
$
    \frac{\delta E_\nu}{\delta \mu}(\mu) = \phi_\mu - \phi_\nu
$. For measures $\mu$ et $\nu$ with finite G-energy, i.e. for measures such that 
$
    \int G d\mu^{\otimes 2}, \int G d\nu^{\otimes 2} < +\infty
$,
this function is locally integrable, and satisfies in a distributional sense: 
\begin{equation}
\label{eq: lag crit eq}
    \nabla \frac{\delta E_\nu}{\delta \mu}(\mu) = \nabla \phi_\mu - \nabla \phi_\nu\,.
\end{equation}

If $\mu$ is a Lagrangian critical point for the Coulomb kernel, this quantity is constant equal to 0. Upon differentiating once again in a distributional sense we conclude that in the interior of the domain $\text{supp}(\mu)$, $d\mu = d\nu$ as per Proposition \ref{InvLaplace}, proving the first part of the theorem.
The proof in the Energy Distance case requires iterating Laplacians.

\begin{proposition}
    Let $\mu$ be a probability measure on $\mathbb{R}^d$. Consider the associated potential 
    $\phi_{\mu}(x) \coloneqq \int -\|x-y\|d\mu(y)$.
    Then, its distributional Laplacian exists almost everywhere and is given by: 
    \begin{equation*}
        \Delta \phi_{\mu}(x) \coloneqq \int -\frac{d-1}{\|x-y\|} d\mu(y)\,.
    \end{equation*}
\end{proposition}

\begin{proof}
    Let $g \in \mathcal{C}_c^\infty$, and denote 
    $P_{\mu} (x)$ as $\int -\frac{d-1}{\|x-y\|} d\mu(y)$.
    We wish to show 
    $
        \langle \phi_{\mu},\Delta g \rangle = \langle P_{\mu}, g \rangle
    $.
    The primary challenge arises from the singularity of the function $\|\cdot - y\|$ at $y$. We circumvent this by integrating on a small ball of radius $\varepsilon >0$ near $y$, carefully controlling the error term. The inversion of integrals in the second equality is justified by the Fubini theorem. 
\begin{align*}
    \langle \phi_{\mu},\nabla g \rangle &\coloneqq \int \phi_{\mu}(x) \Delta g(x) dx \\
    &= \int \left( \int - \|x-y\| \Delta g(x) dx \right) d\mu(y) \\
    &= \int \left[ \left( \int_{B(y,\varepsilon)} -\|x-y\| \Delta g(x) dx \right)  +\left( \int_{\mathbb{R}^d \setminus B(y,\varepsilon)} -\|x-y\| \Delta g(x) dx \right)\right] d\mu(y) \\
     \langle \phi_{\mu},\nabla g \rangle &= \int \left[  I_\varepsilon (y) + J_\varepsilon (y) \right] d\mu(y)\,,
\end{align*}
where the quantities $I_\varepsilon$ and $J_\varepsilon$ are defined by the formulas just above.
First, as $g \in \mathcal{C}_c^\infty$, we get:
\begin{equation*}
    |I_\varepsilon (y)| \leq 2 \| \Delta g\|_{\infty} \varepsilon\,.
\end{equation*}
For $J_\varepsilon (y)$, we can use Green's Formula, since all our quantities are in $\mathcal{C}_c^\infty$ on the open set $ \Omega \coloneqq \mathbb{R}^d \setminus B(y,\varepsilon)$. If we take $\eta(x)$ to be the unit vector at $x \in \partial \Omega(y)$ pointing toward the exterior of $\Omega(y)$, then  $\eta(x) = -\frac{x-y}{\|x-y\|}$. For $x \in \partial \Omega(y)$, the directional derivative of a function $f$ at x is defined by  $\frac{\partial f}{\partial \eta} (x) \coloneqq \langle \nabla f(x) , \eta(x) \rangle$. We denote $dS$ the usual area measure on $B(y,\varepsilon)$. Using Green's Formula we have:
\begin{equation*}
    J_\varepsilon (y) = \int_{\Omega(y)} \Delta -d_y (x) g(x)dx + \int_{\partial \Omega(y)} \frac{\partial d_y}{\partial \eta}(x)g(x)dS(x) - \int_{\partial \Omega(y)} d_y(x)\frac{\partial g}{\partial \eta}(x)dS(x).
\end{equation*}
First, we know that $\Delta d_y(x) = \frac{d-1}{\|x-y\|}$.
Second, as $\nabla d_y(x) = \frac{x-y}{\|x-y\|}$, we have $\frac{\partial d_y}{\partial \eta}(x) = -1$. This leads to:
\begin{equation*}
    \left|\int_{\partial \Omega(y)} \frac{\partial d_y}{\partial \eta}(x)g(x)dS(x) \right|  \leq \|g\|_{\infty} \mathcal{A}(S(y,\varepsilon))\,.
\end{equation*}
Finally, as $g$ vanishes at infinity, there exists a constant $C$ independent of $y$ such that:
\begin{equation*}
    \left|\int_{\partial \Omega(y)} d_y(x)\frac{\partial g}{\partial \eta}(x)dS(x) \right| \leq C \mathcal{A}(S(y,\varepsilon))\,.
\end{equation*}

All these quantities vanish as $\varepsilon \rightarrow 0^+$, independently of $y$, which implies 
$\langle \phi_{\mu},\Delta g \rangle = \langle P_{\mu}, g \rangle$, i.e. 
$\Delta \phi_\mu(x) = \int -\frac{d-1}{\|x-y\|}d\mu(y).$
\end{proof}

Now, we prove the following property, enabling us to prove next the second statement in Theorem \ref{eq measure open}.

\begin{lemma}
    Let $\mu$ be a Lagrangian critical point for the Energy Distance Wasserstein gradient flow towards $\nu$.
    Then, on the open set $\text{Int(supp}(\mu))$ 
    $\phi_\mu^G = \phi_\nu^G$
    where $G$ is the Coulomb kernel.
\end{lemma}

\begin{proof}
    We denote $P_k^{\mu}(x) \coloneqq \int \frac{1}{\|x-y\|^{k}}d\mu(y) $ and prove the following result by finite induction: 
    \begin{equation*}
        \forall k<(d-1)/2, P_{2k+1}^{\mu} = P_{2k+1}^{\nu}\,.
    \end{equation*}
    
    We already established the result is true for $k=1$.
    Now let us take $k<(d-3)/2$ and suppose $P_{2k+1}^{\mu} = P_{2k+1}^{\nu}$.
    We know that if $v : \mathbb{R}^+ \rightarrow \mathbb{R}$ and $u(x) \coloneqq v(\|x\|)$, then its distributional Laplacian is given by: 
    \begin{equation*}
        \Delta u(x) = \frac{d-1}{\|x\|}v'(\|x\|) + v''(\|x\|)\,.
    \end{equation*}
    If $v(r) \coloneqq \frac{1}{r^{2k+1}}$, then 
    $\Delta u(x) = \frac{(2k+1)(2k+3-d)}{\|x\|^{-(2k+3)}}.$
    Denoting $K_y(x) \coloneqq \frac{1}{\|x-y\|^{2k+1}}$,
    we have 
    $\Delta K_y(x) = \frac{(2k+1)(2k+3-d)}{\|x-y\|^{2k+3}}\,.$
    Again, let us take $g \in \mathcal{C}_c^\infty$ and prove 
    $\langle P_{2k+1}^\mu , \Delta g \rangle = (2k+1)(2k+3-d)\langle P_{2k+3}^\mu , g \rangle. $
    We use Fubini's theorem, and avoid the singularity around small balls of radius $\varepsilon >0$: 
    \begin{align*}
    \langle P_{2k+1}^\mu,g \rangle &\coloneqq \int P_{2k+1}^\mu(x) \Delta g(x) dx \\
    &= \int \left( \int K_y(x)\Delta g(x) dx \right) d\mu(y) \\
    &= \int \left[ \left( \int_{B(y,\varepsilon)} K_y(x)\Delta g(x) dx \right)  +\left( \int_{\mathbb{R}^d \setminus B(y,\varepsilon)} K_y(x)\Delta g(x) dx \right)\right] d\mu(y) \\
    \langle P_{2k+1}^\mu,g \rangle &= \int \left[  I_\varepsilon (y) + J_\varepsilon (y) \right] d\mu(y) \,.
\end{align*}
Now, again, we need to control the growth of $I_\varepsilon (y)$ and $J_\varepsilon (y)$.
Since $2k+1<d$ we get:
\begin{equation*}
    \int_{B(y,\varepsilon)} K_y(x)\Delta g(x)dx \leq \|\Delta g(x)\|_{\infty}\int_{B(0,\varepsilon)} K_0(x)dx\,,
\end{equation*} 
and this quantity tends to 0 uniformly as $\varepsilon$ tends to 0. Now, as done previously, we use Green's Formula to express $J_\varepsilon(y)$: 
$$J_\varepsilon (y) = \int_{\Omega(y)} \Delta K_y(x) g(x)dx - \int_{\partial \Omega(y)} \frac{\partial K_y}{\partial \eta}(x)g(x)dS(x) + \int_{\partial \Omega(y)} K_y(x)\frac{\partial g}{\partial \eta}(x)dS(x)\,.$$
We have: 
$$\frac{\partial K_y}{\partial \eta}(x) = \langle -(2k+1)\frac{x-y}{\|x-y\|^{2k+3}}, -\frac{x-y}{\|x-y\|}\rangle = \frac{1}{\|x-y\|^{2k+2}}\,,$$
which gives: 
\begin{equation*}
    \left|\int_{\partial \Omega(y)} \frac{\partial K_y}{\partial \eta}(x)g(x)dS(x)\right| \leq \|g\|_\infty \int \frac{1}{\|x-y\|^{2k+2}}dS(x) = \|g\|_\infty C_d \varepsilon^{d-1-(2k+2)}\,,
\end{equation*}
and again the right term tends to 0 uniformly in $y$ as $\varepsilon$ tends to 0, as $2k+2<d-1$.
Finally: 
$$\int_{\partial \Omega(y)} K_y(x)\frac{\partial g}{\partial \eta}(x)dS(x) \leq \|\frac{\partial g}{\partial \eta}\|_\infty \int_{\partial \Omega(y)}K_y(x)dS(x) = \|\frac{\partial g}{\partial \eta}\|_\infty C_d\varepsilon^{d-1-(2k+1)}\,, $$
which tends to 0 uniformly in $y$ as $\varepsilon$ tends to 0.
This proves: 
\begin{equation*}
    \Delta P_{2k+1}^\mu = (2k+1)(2k+3-d)P_{2k+3}^\mu\,.
\end{equation*}
Now, back to our hypothesis $P_{2k+1}^\mu = P_{2k+1}^\nu$. By taking the distributional Laplacian from both sides, we get $(2k+1)(2k+3-d)P_{2k+3}^\mu = (2k+1)(2k+3-d)P_{2k+3}^\nu$, and as $2k<d-3$, we get: 
$P_{2k+3}^\mu = P_{2k+3}^\nu$.
This proves our result.
Now, to obtain the conclusion, if $d$ is odd, by taking $k = (d-3)/2 < (d-1)/2$, we finally get in the interior of $\text{supp}(\mu)$ the equality $\phi_\mu^G = \phi_\nu^G$.
\end{proof}

The potentials $\phi_\mu^G $ and $\phi_\nu^G$ are superharmonic on the interior of $\text{supp}(\mu)$, and taking their distributional Laplacian yields 
$\mu_{|\text{Int(supp}(\mu))} = \nu_{|\text{Int(supp}(\mu))}\,, $
which is exactly the second statement of Theorem \ref{eq measure open}.

In this section, we were able to characterize a particular class of critical points. However, our result does not capture the singular parts of our measures. In the next section, we study how singular measures behave under the Wasserstein gradient flow of $E_\nu$, proving in Theorem \ref{ThSingMeasCritPts} that singular enough measures cannot be critical points.

\section{No Local Minima in the Wasserstein Geometry}
\label{SecNoLocalMinimum}

Recall that our PL inequality does not hold at a measure $\mu$ if there is no strictly positive lower bound on the density of $\mu$. In particular, it is unclear if local non-global minima of the MMD norm $E_\nu$ can exist. In this section, we establish that there are no local minima other than the global one.
The previous section was dedicated to Lagrangian critical points. We described measures $\mu$ where the velocity field induced by $\nabla \frac{\delta E\nu}{\delta \mu}(\mu) $ is equal to 0 everywhere. In these cases, a JKO step cannot be described through a push-forward map.
However, this pushforward action by maps does not describe all possible dynamics, and diffusion of the mass can occur. In terms of Wasserstein geometry, it means, heuristically, that a descent direction has to be found in the space of velocity plans \cite[Section 12.4]{ambrosio2005gradient}, instead of $L_2$ velocity maps.
To do this, we use concepts inspired by flow interchange techniques developed in \cite{matthes2009flowinter,poissonflowinter2015}. Instead of studying the gradient flow of $E$, we study how $E$ behaves along a certain auxiliary flow, here associated with the Boltzmann entropy functional.

We show, as the main result of this section, the following theorem.

\begin{theorem}[No local non global minima]\label{NoLocMin}
    Let $\mu$ be a probability measure. Then, if $\mu \neq \nu$, there exists a curve $\mu_t$ which is $1/2$-H\"older for the Wasserstein distance, such that $t \mapsto E_\nu(\mu_t)$ is strictly decreasing for $t$ small enough.
\end{theorem}

We begin by noting that the MMD energy only depends on the difference between the measures. Indeed, since $(\mu,\nu) \mapsto E_\nu(\mu)$ is a function of $\mu - \nu$, we can use the Hahn-Jordan decomposition of the signed measure $\mu - \nu$. 

\begin{lemma}[Hahn-Jordan decomposition]
    Let $\mu$ and $\nu$ be two probability measures. Then there exists a unique decomposition $\mu - \nu = \mu_+ - \nu_-$, where $\mu_+$ and $\nu_-$ are mutually singulars measures. Additionally, $\mu_+ \ll \mu$ and $\nu_- \ll \nu$.
\end{lemma}
A heat diffusion process can be applied to $\mu_+$, enabling us to prove the main theorem of this section, stating that even though our functional is non-convex in the Wasserstein geometry, it does not admit any local minima for the Wasserstein geometry that is not global (meaning $\mu = \nu$ or in an equivalently $\mu_+ = \nu_- = 0$). 

\subsection{Heat Diffusion Perturbation}

Let $\rho_0$ be a measure dominated by $\mu$. We write $\mu = \mu - \rho + \rho$ and consider the curve $\rho_t$ defined as the solution of the heat equation $\partial_t \rho_t = \Delta \rho_t$ with initial condition $\rho_0 = \rho$. As is well-known, the heat equation is the Wasserstein gradient flow of the Boltzmann entropy functional defined by: 
$
    \mathcal{H}(\rho) \coloneqq \int \rho \log(\rho) dx
$
if the measure has a density $\rho$ (with abuse of notation) with respect to the Lebesgue measure and $+\infty$ otherwise. Furthermore, the solution $\rho_t$ is explicit: $
    \rho_t = K_t \star \rho$, where $K_t$ is the heat kernel in $\mathbb{R}^d$ defined by 
\begin{equation*}
    K_t(x) \coloneqq \frac{1}{(4\pi t)^{d/2}}\exp (-\|x\|^2/4t)\,.
\end{equation*}
\begin{remark}
    Here, $\rho$ is not necessarily a probability measure. However, we can write the whole Wasserstein formalism for any measure space $\mathcal{M}_m \coloneqq \{ \mu \in \mathcal{M}_+(\mathbb{R}^d), \mu(\mathbb{R}^d) = m \}$. We  denote $W_2(\alpha, \beta)$ the Wasserstein distance on $\mathcal{M}_m$ if $\alpha, \beta \in \mathcal{M}_m$.
\end{remark}

\begin{lemma}\label{differentiability of the energy along the flow}
Consider the curve defined by 
$
    \mu_t \coloneqq \mu + \rho_t - \rho\,,
$
where $\rho_t$ is the heat flow at time $t$ of $\rho_+$. Then,
    the curve $\mu_t$ is absolutely continuous. More precisely, there exists a constant $C>0$ such as, for any $s,t >0$: 
    \begin{equation*}
        W_2(\mu_t,\mu_s) \leq \sqrt{|t-s|} C\,.
    \end{equation*}
\end{lemma}

\begin{proof}
Standard results exposed in \cite[Theorem 11.2.8]{ambrosio2005gradient}  show that $K_t \star \rho$ is absolutely continuous for the Wasserstein metric.
We fix some $s,t >0$, and write the EVI equation associated with the heat flow (see \eqref{EVI_Equation}), that states that for all $\alpha$ density measure with the same mass as $\rho$ and for almost every $t>0$: 
\begin{equation*}
    \frac{1}{2} \frac{d}{dt}W_2^2(\rho_t,\alpha) \leq \mathcal{H}(\alpha) - \mathcal{H}(\rho_t)\,.
\end{equation*}
This implies that there exists a constant $C>0$ such that,for any $s,t >0$ 
$
    W_2(\mu_t,\mu_s) \leq \sqrt{|t-s|} C
$.
Let $\gamma$ be an optimal transport plan $\gamma$ between $K_t \star \rho$ and $K_s \star \rho$, for $s,t > 0$.
The plan $\tilde{\gamma} \coloneqq \gamma + (\mu - \rho)^{\otimes 2}$ is a transport plan between $\mu_t$ and $\mu_s$. Thus: 
\begin{equation*}
    W_2(\mu_t,\mu_s) \leq W_2(K_t \star \rho, K_s \star \rho)\,,
\end{equation*}
which proves the lemma.
\end{proof}
Now, we need some estimates on $E_\nu$ along the flow defined above. It is possible thanks to the next lemma. 
\begin{lemma}\label{diff energy}
    The function $t \mapsto E_\nu (\mu_t)$ is differentiable for every $t>0$, and its derivative is given by: 
    \begin{equation}
    \label{deriv energy heat flow}
        \frac{1}{c_d} \frac{d}{dt}E_\nu(\mu_t) = - \left< \rho, K_{2t} \star \rho + K_t \star (\mu - \rho) - K_t \star \nu \right>\,.
    \end{equation}
\end{lemma}

\subsection{Estimates of Mass Transfers in the Diffusion Process}

In this paragraph, we show that for a well-chosen positive measure $\rho$ and for $t$ small enough, the quantity in Formula \eqref{deriv energy heat flow} is strictly negative.
An essential property is the following lemma \cite{Watson1994}.  

\begin{lemma}[Heat kernel estimates, \cite{Watson1994}]\label{lemma:estim heat singular measures}
    Let $\alpha$ and $\beta$ be two positive mutually singular measures.
    Then, as $t \rightarrow 0$, for $\mu$ almost every $x$: 
    $$K_t \star \beta (x) = o(K_t \star \alpha(x))\,,$$
\end{lemma}
\noindent
We use this lemma to derive the following estimate. 
\begin{lemma}\label{modif estim heat singular measures}
    Let $\alpha$ and $\beta$ be two positive mutually singular measures.
    Then, as $t \rightarrow 0$, for $\mu$ almost every $x$: 
    $$K_t \star \beta (x) = o(K_{2t} \star \alpha(x))\,,$$
\end{lemma}

\begin{proof}
    From the expression of $K_t$, we have 
    $\frac{K_t(x)}{K_{2t}(x)} = 2^{d/2} \exp(-\|x\|^2/8t) \leq 2^{d/2}$,
    from which we deduce 
    $K_t \star \alpha (x) \leq 2^{d/2}K_{2t} \star \alpha (x)$
    and: 
    $$\frac{K_t \star \beta(x)}{K_{2t} \star \alpha(x)} =\frac{K_t \star \alpha(x)}{K_{2t} \star \alpha(x)}  \frac{K_t \star \beta(x)}{K_t \star \alpha(x)} \leq 2^{d/2} \frac{K_t \star \beta(x)}{K_t \star \alpha(x)}\,.$$
     This quantity tends to $0$ from the preceding lemma.
\end{proof}

\subsection{Proof of Theorem \ref{NoLocMin}}

We now prove the main theorem of this section, which is implied by the following proposition.

\begin{proposition}\label{theo:deriv neg mmd}
    Let $\mu$ be a probability measure. We suppose $\mu \neq \nu$, and write $\mu - \nu = \mu_+ - \nu_-$ the unique associated Hahn-Jordan decomposition. 
    Then there exists a $\mu_+$ measurable set $A$ such as $\mu_+(A) \geq \mu_+(\mathbb{R}^d)/2$ and $t_0>0$ such as, for the curve 
    $\mu_t = \mu + K_t \star \mu_{+|A} - \mu_{+|A}$,
    then for all $t<t_0$ we get $\frac{d}{dt}E_\nu(\mu_t) < 0$.
\end{proposition}

\begin{proof}
    The preceding lemma gives the following result: for $\mu_{+}$ almost every $x$, there exists $t_x > 0$ such as, for all $t<t_x$ : $K_t \star \nu_- (x) < \frac{1}{2} K_{2t} \star \mu_+ (x)$.
    Let us consider the set sequence defined for $N >0$ by: 
    $$X_N \coloneqq \left\{ x \in \operatorname{supp}(\mu_+) | t_x \geq 1/N \right\}\,.$$
    This is a growing sequence for inclusion, and it verifies 
    $\operatorname{supp}(\mu_+) = \bigcup_{N \in \mathbb{N}^*} X_N$.
    Now if we write $\widetilde{X_1} \coloneqq X_1 $ $\widetilde{X_{N+1}} \coloneqq X_{N+1} \setminus {X_N}$, we get a countable sequence of disjoint sets whose union is of total mass for $\mu_+$. This implies, by $\sigma$-additivity, that there exists an integer $N_0$ such that:
$$\mu_+\left(\bigcup_{N=1}^{N_0} X_N \right) \geq \mu_+(\mathbb{R}^d)/2\,.$$
    Now, we take $A \coloneqq \bigcup_{N=1}^{N_0} X_N$ and $t_0 = 1/N_0$, which gives, for $t < t_0$: 
    \begin{align}
    \label{ineg deriv}
        \frac{d}{dt}E_\nu (\mu_t) &= -c_d \left< \mu_{+|A},  K_{2t} \star \mu_{+|A} - K_t \star \nu_- \right> \\
        & \leq -\frac{1}{2}c_d \left< \mu_{+|A}, K_{2t} \star \mu_{+|A} \right>  \,.
    \end{align}
    In particular, we get the conclusion $\frac{d}{dt}E_\nu (\mu_t) < 0$.
\end{proof}

\subsection{Dimension of Measure and Critical Points}

In Theorem \ref{NoLocMin}, we saw that for any measure $\mu$ distinct from the target measure $\nu$ we could find an absolutely continuous curve for the Wasserstein distance $\mu_t$ such as $t \mapsto E_\nu(\mu_t)$ was strictly decreasing near 0. However, it is not sufficient to guarantee that the JKO steps \eqref{JKOSteps} do not remain stationary at $\mu$. Our Theorem \ref{eq measure open} states that if $\mu$ is a critical point, then on the interior of its support $\mu$ is equal to $\nu$. We will prove that if the part of $\mu$ singular to $\nu$ is supported on sets singular enough, then $\mu$ cannot be a critical point as in Definition \ref{def: crit point}. In analogy to the finite dimensions and regular case, it corresponds to the fact that it cannot be a second-order critical point.

To quantify the singularity of the support of a measure, we use geometric measure theory properties, mainly the growth of $\mu(B(x,r))$. If it grows faster than $r^d$, we prove that the heat diffusion makes the energy $E_\nu$ decrease fast enough to compensate for the growth of $W_2^2(\mu, K_t \star \mu)$, so that $\mu$ is not a point where JKO steps get stuck.  To do this, we use a more precise version of Lemma \ref{lemma:estim heat singular measures}, see \cite{Watson1994}:

\begin{lemma}\label{lemma:estim heat general}
    Let $\mu$ be a positive measure in $\mathbb{R}^d$ and $q \in [0,n]$. Then, there exists a universal constant $c_{d,q}$ only dependent on $d$ and $q$ such as, for all $x \in \mathbb{R}^d$: 
    \begin{align}
        \underset{r \rightarrow 0}{\liminf}\, r^{-q}\mu(B(x,r)) &\leq c_{d,q} \underset{t \rightarrow 0}{\liminf}\, t^{(d-q)/2} K_t \star \mu (x) \\
        &\leq c_{d,q} \underset{t \rightarrow 0}{\limsup }\, t^{(d-q)/2} K_t \star \mu (x) \leq \underset{r \rightarrow 0}{\limsup }\, r^{-q}\mu(B(x,r))\,.
    \end{align}
\end{lemma}
\noindent
\textbf{How to interpret this result?} This result allows, in some cases, to get precise estimates on the decay of $K_t \star \mu$. 
To illustrate, if $\mu$ has a continuous bounded density function $f$ with respect to the Lebesgue measure, then, with $q = d$, we get that 
$
    \underset{r \rightarrow 0}{\lim}\, r^{-d}\mu(B(x,r)) = f(x)\,,
$
which gives: 
\begin{equation*}
    \underset{t \rightarrow 0}{\lim}\, K_t \star \mu (x) = f(x)/c_{d,q}\,.
\end{equation*}
If $\mu$ charges more singular sets than open sets, for example if there exists some $\delta>0$ such as, for all $x \in A$, where $\mu(A)>0$ 
$
    \underset{r \rightarrow 0}{\lim}\, r^{-(d-\delta)}\mu(B(x,r)) = f(x)
$,
then, for all $x \in A$ we have 
$
    \underset{t \rightarrow 0}{\lim}\, t^{\delta / 2} K_t \star \mu (x) = f(x)/c_{d,q}\,.
$
That is, when $t \rightarrow 0$: 
\begin{equation*}
    K_t \star \mu (x) \sim \frac{f(x)}{c_{d,q}} \frac{1}{t^{\delta/2}}\,.
\end{equation*}


Let $\mu$ and $\nu$ be two probability measures with finite Coulomb energy. Once again we write the Hahn-Jordan decomposition $\mu - \nu = \mu_+ -\nu_-$. 
From the result of Section \ref{SecCriticalPoints}, we know that at a critical point that is not the global minimum, $\mu_+$  has an empty interior support. 

We study what happens if $\mu_+$ is singular in the sense that the growth of $\mu_+(B(x,r))$ on the support of $\mu_+$ as $r$ tends to 0.
We prove the following lemma: 

\begin{lemma}\label{SingMeasureCrit}
    If there exists a set $A \subset \mathbb{R}^d$ such that $\mu_+(A)>0$, some positive number $0<\delta<2$, a constant $C > 0$ and $r_0>0$ such that, for all $0\leq r \leq r_0$ and $\mu_+$ almost every $x \in A$: 
    $$\mu_+(B(x,r)) \geq C r^{d-\delta}\,,$$
    then $\mu$ is not a critical point as in Definition \ref{def: crit point}.
\end{lemma}

\begin{proof}
    In this proof, we denote by $C$ any strictly positive constant that does not depend on $t$ or $\tau$.
    Once again, we consider the Wasserstein curve defined by $\mu_0 = \mu$ and 
    $\mu_t = \mu + K_t \star \mu_{+|A} - \mu_{+|A}$. 
    Using Lemma \ref{lemma:estim heat general}, we get the existence of $t_0 > 0$ such that, for all $0<t<t_0$ and $\mu_+$ almost every $x \in A$
    \begin{equation}
    \label{decay heat density}
        K_{2t} \star \mu_+(x) \geq \frac{C}{t^{\delta/2}}\,.
    \end{equation}
    Moreover, by the same reasoning as in the proof of Theorem \ref{theo:deriv neg mmd}, we can suppose, if we take a subset of $A$, that there exists $t_1>0$ such as, for any $0<t<t_1$: 
    \begin{equation*}
        \frac{d}{dt}E_\nu(\mu_t) \leq - C\left<\mu_{+|A}, K_{2t} \star \mu_{+|A} \right>\,.
    \end{equation*}
    Now, using Formula \eqref{decay heat density}, this gives along $\mu_t$,  
    $
        \frac{d}{dt}E_\nu(\mu_t) \leq - Ct^{-\delta/2}
    $.
    From the EVI inequality associated with the entropy gradient flow, we get the existence of a constant $C$ such that
    $
        \frac{d}{dt}W_2^2(\mu_0,\mu_t) \leq C
    $.
    Combining these inequalities, we get an estimate on the derivative of our proximal functional $E_\nu^{\tau, \mu} \coloneqq E_\nu + \frac{1}{2\tau}W_2^2(\mu_0, \cdot)$: 

    \begin{equation*}
        \frac{d}{dt} E_\nu^{\tau, \mu} (\mu_t) \leq C \left(-t^{-\delta/2} + \frac{1}{2\tau}\right)\,.
    \end{equation*}
    This provides, integrating from $0$ to $t$, as $\delta/2<1$: 
    \begin{equation}
    \label{var fct prox}
        E_\nu^{\tau, \mu}(\mu_t) - E_\nu^{\tau, \mu}(\mu_0) \leq C \left(-\frac{2t^{-\delta/2 + 1}}{2-\delta} + \frac{t}{2\tau}\right)\,.
    \end{equation}
    Optimizing this quantity over $t>0$, we get an optimal 
    $t_\tau \coloneqq (2\tau)^{2/\delta}$, smaller than $t_0, t_1$ if $\tau$ is small enough so that all our inequalities are true for $t_\tau$. Injecting $t_\tau$ in equation \eqref{var fct prox} we get: 
    \begin{equation*}
        E_\nu^{\tau, \mu}(\mu_{t_\tau}) - E_\nu^{\tau, \mu}(\mu_0) \leq -C \tau^{\frac{2-\delta}{\delta}} \,,
    \end{equation*}
    concluding the proof since:
    \begin{equation*}
        \underset{\rho \in \mathcal{P}_2(\mathbb{R}^d)}{\min} \, E_\nu^{\tau, \mu}(\rho)- E_\nu^{\tau, \mu}(\mu_0) \leq \underset{t >0}{\min} \, E_\nu^{\tau, \mu}(\mu_t)- E_\nu^{\tau, \mu}(\mu_0) \leq -C \tau^{\frac{2-\delta}{\delta}} \, .\qedhere
    \end{equation*}
\end{proof}
This lemma is used to prove the main result of this subsection.
\begin{theorem}\label{ThSingMeasCritPts}
    Let us assume there exists a set $A$ such that $\mu_+(A)>0$ and such that $\mu_{+|A}$ has Minkowski dimension less than $d-\delta$ (see \cite{Heurteaux2007}) for some $0<\delta\leq 1$.
    Then $\mu$ is not a critical point as in Definition \ref{def: crit point}.
\end{theorem}

\begin{proof}
    The local Minkowski dimension of a measure at $x$ is defined, if the limit exists, as
    \begin{equation*}
\operatorname{dim}_{\mathcal{M}}^x(\mu) \coloneqq \underset{r \rightarrow 0}{\lim }\, \frac{\log(\mu(B(x,r)))}{\log(r)}\,.
    \end{equation*}
    If $\mu_+$ charges some set $A$ such as, for all $x \in A$, $\operatorname{dim}_{\mathcal{M}}^x(\mu)$ exists and is less than $d-\delta$, then for every $\varepsilon > 0$ there exists a radius $r_x$ such that for any $r<r_x$, $\mu(B(x,r)) \geq r^{d-\delta + \varepsilon}$.
    Using the same $\sigma$-additivity arguments as in the proof of Proposition \ref{theo:deriv neg mmd}, we prove we can choose a uniform $r_0$ by considering a subset of $A$ of non-zero measure for $\mu_+$.
    This is exactly the hypothesis of Lemma \ref{SingMeasureCrit}, concluding the proof.
\end{proof}
\begin{remark}
    The condition on the Minkowski dimension is satisfied if, for example, $\mu_+$ is absolutely continuous with respect to the volume measure on some manifold of dimension $d-1$.\\
    Note that the result is stated only for $\delta \leq 1$ instead of $\delta < 2$ (as assumed in the lemma). This is explained by the fact that for $\delta > 1$, the energy functional is infinite. In such a case, it is straightforward to prove from the very definition that such a measure (of infinite energy) cannot be a critical point.
\end{remark}

\par{\textbf{Comments on possible critical points.}}
Our main goal in this section was to prove that there is no local minima apart from the global one. However, it does not prevent the existence of critical points. Passing by, we were able to rule out some singular measures from being critical points. Indeed, we proved that if the current measure has a Minkowski dimension strictly less than the ambient one, then it cannot be a critical point as defined in Section \ref{SecCriticalPoints}. However, there are still many candidates for critical points such as measures proportional to $\mathbbm{1}_{A}$ where $A$ is a positive Lebesgue measure closed set of empty interior such as a fat Cantor set.
A complete characterization of critical points, and the existence of a nontrivial example, is left open for future works.

\subsection{Extension to Riemannian Manifolds}

Recall that the heat kernel is defined on a general Riemannian manifold.

\begin{proposition}[Heat kernel on a manifold]
    Let $(M,g)$ be a Riemannian manifold (it doesn't need to be compact). Then we can define its heat kernel $K : (0,\infty) \times M \times M \rightarrow \mathbb{R}$ as the smallest positive fundamental solution of the heat equation, meaning that for any $y \in M$:
    \begin{equation*}
    \left\{
        \begin{array}{ll}
             \partial_t K = \Delta_x K  \\
             K(t,\cdot,y) \underset{t \rightarrow 0}{\rightarrow} \delta_y \,.
        \end{array}
        \right.
    \end{equation*}
\end{proposition}
To generalize the conclusions of Section \ref{SecNoLocalMinimum}, we need estimates resembling the one in the proof of Lemma \ref{modif estim heat singular measures}, mainly that there is some constant $C > 0$ such that: 
\begin{equation}
\label{eq:heat kernel growth}
    K_{2t} \leq C K_t\, .
\end{equation}
In the Euclidean case $M = \mathbb{R}^d$, we got $C = 2^{d/2}$. We do not detail the conditions for this to hold, but it is true if $M$ is a nilpotent Lie group equipped with a left-invariant metric (\cite{VAROPOULOS19901}), if $(M,g)$ is a geodesically complete non-compact Riemannian manifold of nonnegative Ricci curvature (\cite{grigoryan}), or if $(M,g)$ is compact (\cite{li1986parabolic}).
For example, on the flat 1-dimensional torus $\mathbb{T} \coloneqq \mathcal{S}^1$ which we represent as $[0, 1]/\{ 0 \sim 1 \} = \mathbb{R}/\mathbb{Z}$, the heat kernel is given by the periodization of the Euclidean heat kernel:
\begin{equation*}
    K_\mathbb{T}(t,x,y) := \sum_{n \in \mathbb{Z } }K_\mathbb{R}(t,x + k,y)\,.
\end{equation*}
As the flat d-dimensional torus $\mathbb{T}^d \coloneqq (\mathcal{S}^1)^d$ is a product space, its heat kernel is the product, defined for $x = (x_1, ..., x_d)$ and $y = (y_1, ..., y_d)$ by:
\begin{equation*}
    K_{\mathbb{T}^d}(t,x,y) = \prod_{1 \leq k \leq d} K_\mathbb{T}(t,x_k,y_k)\,,
\end{equation*}
so that the desired estimate in Formula \eqref{eq:heat kernel growth} holds for some $C > 0$.
An example of a manifold where it does not hold is the hyperbolic space.

\newpage
\clearpage
\markboth{}{}
\section*{Acknowledgements}
The authors thank very warmly Stephen C. Preston for his help on radial solutions. They also thank Adrien Vacher for his help at the beginning of this work and Théo Dumont for his careful proofreading.
This work was supported by the Bézout Labex (New Monge Problems), funded by ANR, reference ANR-10-LABX-58.

\printbibliography
\clearpage
\markboth{APPENDIX}{APPENDIX}
\newpage
\appendix
\clearpage
\markboth{APPENDIX}{APPENDIX}

\section{Wasserstein Gradient Flows}

Let us recall the definition of the Wasserstein distance $W_2$ \cite{santambrogio2015optimal,ambrosio2005gradient}.

\begin{definition}
    Let $(X,d)$ be a metric space. We denote $\mathcal{P}_2(X)$ as the set of probability measures on $X$ with bounded second moment, i.e. such that $\int d(x,x_0)^2d\mu(x) < \infty$ for some $x_0 \in X$. The 2-Wasserstein distance $W_2$ between two measures $\mu,\nu \in \mathcal{P}_2(X)$ is defined by:
    \begin{equation}
        W_2^2(\mu,\nu) \coloneqq \underset{\gamma \in \Gamma(\mu,\nu)}{\min} \, \int d(x,y)^2d\gamma(x,y) \, ,
    \end{equation}
    where, if we denote $\pi_i$ as the projection on the i-th coordinate, $\Gamma(\mu,\nu) = \{ \gamma \in  \mathcal{P}_2(X \times X) \, | \, \pi_1 \# \gamma = \mu, \pi_2 \# \gamma = \nu\}$ is the space of transport plans between $\mu$ and $\nu$.
\end{definition}

In this paper, we only consider $X$ to be the Euclidean space or a Riemannian manifold. In these cases, the metric space $(\mathcal{P}_2(X), W_2)$ is a complete geodesic space. The following results are about some properties of continuous curves in $\mathcal{P}_2(X)$  \cite{santambrogio2015optimal,ambrosio2005gradient}. 

\begin{definition}[Continuity equation]\label{ContEq}
    Let $\mu_t$ be a time-indexed family of measures on $\mathbb{R}^d$ and $v_t$ a time-dependent $\mu$-integrable vector field. The curve $\mu_t$ is- said to satisfy the continuity equation associated with $v_t$ if: 
    \begin{equation*}
        \partial_t \mu_t + \nabla \cdot (\mu_t v_t) = 0 \, ,
    \end{equation*}
    in the sense of distributions.
\end{definition}
The continuity equation above means that along the curve $\mu_t$, mass moves along the vector field $v_t$.

\begin{definition}[Absolutely continuous curves]
    A curve of measures $\mu_t$ defined for $t \in [a,b]$ is said to be absolutely continuous if there exists some function $m \in L_2(\mathbb R)$ such that for all $a \leq s \leq t \leq b$:
    \begin{equation*}
        W_2(\mu(s),\mu(t)) \leq \int_s^t m(\tau)d\tau \,.
    \end{equation*}
\end{definition}

\begin{proposition}\label{AbsContVect}
    A curve $\mu_t$ satisfies a continuity equation if and only if it is absolutely continuous.
\end{proposition}

This structure allows some Riemannian-like calculus on the geodesic space $(\mathcal{P}_2(X), W_2)$ as introduced by Otto \cite{OttoPorousMedium}. Indeed, in Definition \ref{ContEq}, the vector field $v_t$ is to be seen as the time derivative of the curve $\mu_t$, taking value in the tangent space at $\mu_t$ for the Wasserstein metric. In some diffusive cases, the time derivative cannot be described through the action of a single vector field at all times but rather, informally, a multivalued vector field taking value in a product space. This is why plans are used in some definitions below (see \cite[12.4]{ambrosio2005gradient} for more details on the geometry of Wasserstein spaces).
Let us introduce subdifferentials in the Wasserstein space \cite[Chap 10]{ambrosio2005gradient}.

\begin{definition}[Extended Fréchet subdifferential]\label{VectSubDiff}
    Consider a probability measure $\mu$ and a functional $\mathcal{F}: \mathcal{P}(\mathbb R^d) \to \mathbb R$. A plan $\gamma \in \mathcal{P}(\mathbb{R}^d \times \mathbb{R}^d)$ belongs to the Fréchet subdifferential $\bm{\partial}\mathcal{F}(\mu)$ of $\mathcal{F}$ at $\mu$  if $\pi_1 \#\gamma = \mu$ and if for every probability measure $\rho$:
    \begin{equation}
    \label{PlanSubdiff}
        \mathcal{F}(\rho) - \mathcal{F}(\mu) \geq \underset{\alpha \in \Gamma(\gamma,\rho)}{\inf} \, \int_{X^3} \langle x_2,x_3-x_1 \rangle d\alpha + o(W_2^2(\mu,\rho)).
    \end{equation}
\end{definition}
In some cases, a transport plan $\gamma \in \bm{\partial}\mathcal{F}(\mu)$ may be concentrated on the graph of a vector field, being of the form:
\begin{equation*}
    \gamma = (Id \times \xi) \# \mu \,,
\end{equation*}
for a vector field $\xi \in L_2(\mu)$. Consequently, the subdifferential $\partial\mathcal{F}(\mu)$ is defined as follows: 

\begin{definition}\label{VecSubDiff}
    Let $\mathcal{F}$ be a functional on $\mathcal{P}_2$ and $\mu \in \mathcal{P}_2$. A vector field $\xi \in L_2(\mu)$ belongs to the subdifferential $\partial\mathcal{F}(\mu)$ of $\mathcal{F}$ at $\mu$ if for every probability measure $\rho$
    \begin{equation*}
        \mathcal{F}(\rho) - \mathcal{F}(\mu) \geq \underset{\gamma_0 \in \Gamma_0(\mu,\rho)}{\inf} \, \int \xi(x) \cdot (y-x) d\gamma_0(x,y) + o(W_2^2(\mu,\rho))\, ,
    \end{equation*}
    where $\Gamma_0(\mu,\rho)$ is the set of optimal transport plans between $\mu$ and $\rho$.
\end{definition}

For a general functional $\mathcal{F}$ and a general probability measure $\mu$, both $\partial\mathcal{F}(\mu)$ and $\bm{\partial}\mathcal{F}(\mu)$ may be empty.

There are several ways to approach gradient flows in the Wasserstein space. The first one is a direct analogy of gradient flows on a manifold using the previous definitions.

\begin{definition}[Pointwise differential formula]\label{PointGrad}
    Let $\mu_t$ be an absolutely continuous curve in $\mathcal{P}_2(\mathbb{R}^d).$ By Proposition \ref{AbsContVect} it is a weak solution of a continuity equation with a time-dependent vector field $v_t$.
    The curve $\mu_t$ is said to be a gradient flow a functional $\mathcal{F}$ if for almost any $t>0$:
    \begin{equation}
    \label{PointGradEq}
        v_t \in -\partial\mathcal{F}(\mu_t).
    \end{equation}
\end{definition}

This definition is quite strong, and such curves may not exist at all. To construct them, the idea is to use a discrete algorithm approximation.
Let $\mathcal{F}$ be a functional defined on $\mathcal{P}$. Let us fix an initial measure $\mu_0$, $\tau >0$ and consider the following discrete recursive scheme, called Minimizing Movement or JKO  steps (\cite{ambrosio2005gradient,jordan1998variational}): 
\begin{equation}
\label{JKOsteps}
    \mu_{k+1}^\tau \in \underset{\rho \in \mathcal{P}_2}{\argmin}\,\mathcal{F}(\tau,\mu_k;\rho) \coloneqq \mathcal{F}(\rho) + \frac{1}{2\tau} W_2^2(\rho,\mu_k^\tau)\,.
\end{equation}
It consists in updating $\mu_{k+1}$ using the proximal function $\mathcal{F}(\tau,\mu_k;\cdot)$.
We assume that the associated sequence $(\mu_k^\tau)_{k\in \mathbb{N}}$ can be constructed and we consider the piecewise constant curves defined, if $t \in [k\tau,(k+1)\tau[$ by 
$
    U_\tau(t) = \mu_k^\tau
$.

\begin{definition}[Minimizing movement curve]
    Let $\mu_0$ an initial probability measure and $\mathcal{F}$ a functional such that the associated sequence $(\mu_k^\tau)_{k\in \mathbb{N}}$ can be built.
    A curve $\mu_t$ is said to be a minimizing movement curve if there exists some sequence $(\tau_k) \searrow 0$ such that $(U_{\tau_k})$ converges narrowly to $\mu_t$. 
\end{definition}
There is no a priori guarantee of uniqueness in the definition above: to a functional $\mathcal{F}$ and an initial probability measure $\mu_0$ can correspond an infinity of minimizing movements. 

Let us make the following assumptions on $\mathcal{F}$:
\begin{enumerate}
    \item $\mathcal{F}$ is proper (not everywhere $+\infty$) and lower semi continuous for the weak topology.
    \item Coercivity: there exists some $\tau_0>0$ such that for all $\tau_0>\tau >0$ and $\mu \in \mathcal{P}_2$, there exists some probability measure $\mu_\tau$ minimizing the proximal function $\mathcal{F}(\tau,\mu;\cdot)$.
\end{enumerate}
Using \cite[Proposition 11.1.6]{ambrosio2005gradient}, we get:

\begin{theorem}\label{ExistMMCurve}
    With the above assumptions, a Minimizing Movement curve always exists.
\end{theorem}
Why are we interested in such curves? The iterations of the discrete scheme \eqref{JKOsteps} satisfy an important regularity property, they are point of subdifferentiability of $\mathcal{F}$ \cite[Th 10.3.4; Remark 10.3.5]{ambrosio2005gradient}.
\begin{proposition}\label{DiffJKO}
    Let $\mu$ be a probability measure and $\mathcal{F}$ a functional such that we can define $\mu_\tau$ as an iteration of \eqref{JKOsteps} from $\mu$.
    Then $\partial\mathcal{F}(\mu_\tau)$ is not empty.
\end{proposition}
Let $\hat{\gamma}_\tau \in \Gamma_0(\mu_\tau,\mu)$. The rescaled velocity plan $\gamma_\tau \coloneqq f_\tau \# \hat{\gamma}_\tau$, where $f_\tau(x_1,x_2) = (x_1,\frac{x_2-x_1}{\tau})$, is in the extended Fréchet subdifferential $\bm{\partial}\mathcal{F}(\mu)$. Moreover, there exists a unique optimal plan such that its barycenter projection $\Tilde{\gamma}_\tau^0$ is in the subdifferential $\partial \mathcal{F}(\mu)$. It is characterized by the strictly convex minimum condition:
\begin{equation*}
    \|\Tilde{\gamma}_\tau^0\|_{L_2(\mu_\tau)} = \underset{\hat{\gamma}_\tau \in \Gamma_0(\mu_\tau,\mu)}{\min}\,\|\Tilde{\gamma}_\tau - Id\|_{L_2(\mu_\tau)}\,.
\end{equation*}
\begin{remark}
    The differentiation point is $\mu_\tau$, not directly $\mu$. This result is to be compared with the fact that in the Euclidean space, the iteration of the implicit gradient descent scheme $x_{k+1} = x_k - \tau \nabla f(x_{k+1})$ can be obtained as: 
\begin{equation}\label{JKOSteps}
    x_{k+1} \in \argmin_y \, f(y) + \frac{1}{2\tau}\|y-x_k\|^2 \,.
\end{equation}
\end{remark}
Due to Proposition \ref{DiffJKO}, we would like to pass to the limit as $\tau \rightarrow 0$ and conclude that a Minimizing Movement curve is a gradient flow in the sense of Definition \ref{PointGrad}. However, this is not always the case and such a curve only satisfies a relaxed gradient equation, with the time-dependent vector field $v_t$ only belonging to the limiting subdifferential of $\mathcal{F}$ at $\mu_t$ \cite[Def 11.1.5]{ambrosio2005gradient}.
In the case of a functional $\mathcal{F}$ which is $\lambda$-convex along generalized geodesics \cite[Def 9.2.2, 9.2.4; Th 11.2.1]{ambrosio2005gradient}, more can be said.
\begin{theorem}[Gradient flow for $\lambda$-convex functionals]\label{GradFlowConv}
    Let $\mathcal{F}$ be a $\lambda$-convex functional along generalized geodesics and $\mu_0 \in \mathcal{P}_2$.
    Then:
    \begin{itemize}
        \item There exists a unique Minimizing Movement curve starting from $\mu_0$.
        \item This limiting curve $\mu_t$ is a solution of the gradient flow equation \eqref{PointGradEq}.
        \item (EVI) The curve $\mu_t$ satisfies the EVI inequality, for all $\nu \in \mathcal{P}_2$
        \begin{equation}\label{EVI_Equation}
            \frac{1}{2}\frac{d}{dt} W_2^2(\mu_t,\nu) \leq \mathcal{F}(\nu) - \mathcal{F}(\mu_t) - \frac{\lambda}{2}W_2^2(\mu_t,\nu)\,.
        \end{equation}
        \item If $\lambda > 0$, $\mathcal{F}$ admits a unique minimum $\mu_*$ and both $\mu_t$ and $\mathcal{F}(\mu_t)$ converge exponentially respectively to $\mu_*$ and $\mathcal{F}(\mu_*)$.
        \item If $\lambda = 0$ and $\mathcal{F}$ admits a minimum $\mathcal{F}_*$ then
        \begin{equation*}
            \mathcal{F}(\mu_t) - \mathcal{F}_* \leq \frac{W_2(\mu_0,\mu_*)}{2t} \, .
        \end{equation*}
    \end{itemize}
\end{theorem}

\begin{example}

    \begin{itemize}
        \item If $\mathcal{F}$ is defined on measures with density with respect to the Lebesgue measure ($\mu = \rho(x)dx$) by the formula $\mathcal{F}(\mu = \rho(x)dx) \coloneqq \int F(\rho)dx$ where $F$ is convex, superlinear, satisfies $F(0)=0$, and that the map $s \mapsto s^dF(s^{-d})$ is convex and non increasing, then $\mathcal{F}$ is convex along generalized geodesics \cite[Th 7.28]{santambrogio2015optimal}.
        For example, this condition is verified for $H : x \mapsto x \log(x)$ which gives the Boltzmann entropy $\mathcal{H}$, and $F_p : x \mapsto x^p$ where $p>1$, defining p-energies.
        \item If $V : \mathbb{R}^d \rightarrow \mathbb{R}$ is $\lambda$-convex, the associated potential functional $\mathcal{V} : \mu \mapsto \int V d\mu$ is $\lambda$-convex along geodesics.
        \item If $W : \mathbb{R}^d \times \mathbb{R}^d \rightarrow \mathbb{R}$ is $\lambda$-convex, the associated auto-interaction functional $\mathcal{W} : \mu \mapsto \int W d\mu^{\otimes 2}$ is $\lambda$-convex along geodesics.
        \item The previous condition is not necessary. In dimension 1, let us define $\Delta_+ \coloneqq \{(x,y), x\leq y \}$ and $\Delta_- \coloneqq \{(x,y), x\geq y \}$. Then if $W$ is convex when restricted to $\Delta_+$ and $\Delta_-$, the associated energy $\mathcal{W}$ is convex along generalized geodesics. That is the case for the Energy Distance kernel $(x,y) \mapsto -\|x-y\|$.
    \end{itemize}

\end{example}

\begin{corollary}[Convergence of the Energy Distance gradient flow in 1D]
    Let $\mu_0, \nu \in \mathcal{P}_2(\mathbb{R})$ with finite first moment.
    Using the previous results, there exists a unique solution $\mu_t$ to the associated gradient flow equation \eqref{PointGradEq} for the functional $E_\nu$ with the energy distance kernel, obtained as a Minimizing Movement curve from $\mu_0$. Furthermore:
    \begin{equation*}
        E_\nu(\mu_t) \leq \frac{W_2(\mu_0,\nu)}{2t} \, .
    \end{equation*}
\end{corollary}

\begin{proof}
    Let $x,y \in \mathbb{R}$.
    If $(x,y) \in \Delta_+$, then $-\|x-y\| = x-y$ and the Energy distance kernel is convex on $\Delta_+$.
    If $(x,y) \in \Delta_-$, then $-\|x-y\| = y-x$ and the Energy distance kernel is convex on $\Delta_-$.
    This proves that $E_\nu(\cdot)$ is convex along Wasserstein geodesics.
    The result directly follows from Theorem \ref{GradFlowConv}.
\end{proof}
However, it is important to note that in higher dimensions, the energy is not convex along generalized geodesics.

\section{Kernels, MMD and Potential Theory}

In this paper, we are interested in particular functionals on probability measure spaces: MMD-energies.

\begin{definition}
    Let $G$ be a conditionally positive kernel (possibly taking infinite values). 
    \begin{itemize}
        \item The associated internal $G$-energy functional of a signed measure $\rho$ is defined by:
    \begin{equation*}
        E(\rho) \coloneqq \frac{1}{2} \langle \rho, G \star \rho \rangle.
    \end{equation*}
        \item If $\mu$ and $\nu$ are probability measures of finite internal energy, we define the functional: 
        \begin{equation*}
        E_\nu(\mu) \coloneqq E(\mu - \nu).
    \end{equation*}
    The Maximum Mean Discrepancy of kernel $G$ between $\mu$ and $\nu$ is defined as: 
    \begin{equation*}
        MMD(\mu,\nu) \coloneqq \sqrt{E(\mu-\nu)}.
    \end{equation*}

    \end{itemize}
    
\end{definition}

Maximum Mean Discrepancies have been studied in the context of Wasserstein gradient flows, but mainly in smooth cases \cite{arbel2019maximum}. 
\subsection{In the Euclidean Space}
A particular class of singular kernels o Riesz kernels.

\begin{definition}[Riesz kernels]
    In the Euclidean space $\mathbb {R}^d$, $d \in \mathbb N$, Riesz kernels are defined by:
\begin{equation*}
    k_s(x,y) = \frac{1}{s\|x - y\|^s}\,,
\end{equation*}    
for $s \in [-1,d-2]\setminus \{0\}$. If $s=0$, the kernel is defined as
$
    k_0(x,y) = -\log({\|x - y\|})\,.
$
\end{definition}
In this article, we  focus on two kernels: 

\noindent
 \textbf{The Coulomb kernel. } In $\mathbb{R}^d$, the Coulomb Kernel is given by 
    $
        k_{d-2}(x,y) = \frac{1}{(d-2)\|x - y\|^{d-2}}
    $.
    It is remarkable as a fundamental solution of the Laplace equation in $\mathbb{R}^d$ in the following sense.
    \begin{proposition}[Corresponding differential operator]\label{InvLaplace}
    Let $G$ be the Coulomb Kernel in $\mathbb{R}^d$. There exists a positive constant $c_d$ such that $G$ is the fundamental solution of $\frac{1}{c_d}(-\Delta)^{\frac{d-s}{2}}$ in $\mathbb{R}^d$, which is, if $\Delta_x G(x,y)$ is equal to the Laplacian of $x \mapsto G(x,y)$:
    \begin{equation*}
        \frac{1}{c_d}(-\Delta_x) G(x,y) = \delta_y\,.
    \end{equation*}
\end{proposition}

\begin{definition}\label{Potential}
    For a kernel $G$ and a measure $\rho$, we define the associated potential by
$
        \varphi_\rho^{G} \coloneqq G \star \mu
$
    when it is well-defined.
\end{definition}

When there is no ambiguity on the kernel $G$ or the measure $\rho$ we will only write $\varphi$ instead of $\varphi_{\rho}^G$.
In the Coulomb case, this potential corresponds to the electric field generated by a distribution $\mu$ of electric charges.

\begin{proposition}[Coulomb kernel and Laplace equation]
    Let $G$ be the Coulomb kernel in $\mathbb{R}^d$ and $\mu$ be a positive measure.
    Then, the associated potential $\varphi_\mu$ satisfies the equation:
    \begin{equation*}
        \frac{1}{c_d}(-\Delta)\varphi_\mu = \mu\,.
    \end{equation*}
\end{proposition}

This equation establishes that, for the Coulomb kernel, the potential $\varphi_\mu$ is harmonic outside of the support of $\mu$, and superharmonic in $\mathbb{R}^d$ \cite{papadimitrakis}.

\noindent
\textbf{The Energy Distance kernel. } The Energy Kernel is independent of the dimension and defined by
    $
        k_{-1}(x,y) = -\|x-y\|
    $.
    It has the good property that any measure with a finite first moment has finite internal energy. This is to compare with the Coulomb kernel, where additional regularity is required, e.g. finite $H_{-1}$ norm.
\begin{proposition}[Convexity on probability measures]
    Let G be the Coulomb or Energy Distance kernel and $\nu \in \mathcal{P}(\mathbb R^d)$ with finite $G$-energy. Then the associated functional $E_\nu(\cdot) = E(\cdot - \nu) $
    is a quadratic functional (in $\mu$) that is strictly convex (for the convex structure on $\mathcal{P}(\mathbb R^d)$) on its domain. It is positive and equal to $0$ if and only if $\mu = \nu$. Moreover, the application $(\mu,\nu) \mapsto MMD(\mu,\nu) \coloneqq \sqrt{E(\mu-\nu)}$ defines a distance on probability measures with finite $G$-energy
\end{proposition}

\subsection{Extension to Riemannian Manifolds}

We focus on compact manifolds without boundaries to avoid possible loss of mass at infinity in the dynamics.
Coulomb-like interactions among probability measures can be defined through the fundamental solution of the Laplace equation on some Riemannian manifolds.
We use the formalism developed in \cite{garcia2019} and \cite{STEINERBERGER2021109076}, along with results from \cite{aubun1998}. 

Let $(M,g)$ be a compact oriented n-dimensional Riemannian manifold without boundary where $g$ is the Riemannian metric. We denote $\pi$ its volume form, assuming $\pi(M) = 1$. Additionally, $\Delta_M : \mathcal{C}^{\infty}(M) \rightarrow {C}^{\infty}(M)$ is the associated Laplace-Beltrami operator. 

\begin{definition}[Green's function on a manifold]
    A kernel $G : M \times M \rightarrow ]-\infty, + \infty ]$ is said to be a Green function if it is symmetric, if $G_x : y \in M \mapsto G(x,y)$ is integrable for all $x \in M$, and if it satisfies the Laplacian equation: 
    \begin{equation}
    \label{lap-bel eq}
        -\Delta_M G_x = -\delta_x + 1\,,
    \end{equation}
    in a distributional sense.
\end{definition}
The following existence and quasi-uniqueness result is given in \cite[Chapter 4]{aubun1998}.
\begin{proposition}
    Let $(M,g)$ as defined previously. Then, equation \eqref{lap-bel eq} admits a unique solution, up to an additive constant.
\end{proposition}
In particular, if $\mu$ is a measure on $M$, we can define the potential 
$
    \varphi_\mu(\mu)(x) \coloneqq \int G(x,y)d\mu(y)
$.
Then, in the sense of distributions 
$
    \Delta_M \varphi_\mu = -\mu + \mu(M)\pi
$.
In particular, $G$ is bounded from below and $\int G(x,y)d\pi(y)$ does not depend on $x$.
We denote by $G$ the unique Green function such that $\int G(x,y)d\pi(y) = 0$. The Green's function $G$ is lower semi-continuous, so that the functional defined by: 
\begin{equation}
\label{coulomb energy mani}
    \mathcal{W}(\mu) = \iint_{M\times M} G(x,y)d\mu(x)d\mu(y)
\end{equation}
is lower semi-continuous for the weak topology.
Furthermore, the kernel $G$ is $\mathcal{C}^\infty$ outside of the diagonal $\mathcal{D} \coloneqq \{(x,x) \, | \, x \in M \}$.
Now, let $\nu$ be a density probability measure on $M$. The energy functional is defined similarly:
\begin{equation*}
    E_\nu(\mu) \coloneqq \frac{1}{2}\iint_{M\times M} G(x,y)d(\mu-\nu)^{\otimes 2}(x,y)\,.
\end{equation*}
This energy is lower semi-continuous, positive, and equal to $0$ if and only if $\mu = \nu$. As in the Euclidean case, the square root is a distance between $\mu$ and $\nu$ \cite{Cartan1945,garcia2019}.

\section{Some Proofs}

\begin{proof}[Proof of Lemma \ref{MinimNormSubdiff}]\label{proof lemma 1}
    Let us fix a vector field $v \in \mathcal{C}_c^{\infty}$, and prove the following equality in both cases: 
    \begin{equation}
    \label{lim deriv direc}
        \underset{t \rightarrow 0}{\lim }\, \frac{E_\nu((id + tv)\#\mu) - E_\nu(\mu)}{t} = \int (\nabla G) \star (\mu - \nu) \cdot v d\mu\,.
    \end{equation}
\noindent
\textbf{First case: } If $G = -\|.\|$.
First let us remark that if $x = y, G(x-y + t(v(x) - v(y))) = G(x-y) = 0$. Then:

\begin{align*}
\MoveEqLeft[1]
    \underset{t \rightarrow 0}{\lim }\, \frac{1}{2} \int  \frac{G(x-y + t(v(x) - v(y))) - G(x-y)}{t} d\mu(x) d\mu(y) \\
    &= \underset{t \rightarrow 0}{\lim } \,\frac{1}{2} \int_{x \neq y}  \frac{G(x-y + t(v(x) - v(y))) - G(x-y)}{t} d\mu(x) d\mu(y) \\
    &= \frac{1}{2} \int_{x \neq y}  \nabla G (x-y) \cdot (v(x)-v(y)) d\mu(x) d\mu(y) \\
    &= \int_{x \neq y}  \nabla G (x-y) \cdot v(x) d\mu(x) d\mu(y) \,. \\
\end{align*}
Same for the other term of the MMD energy: 
\begin{align*}
\MoveEqLeft[1]
    \underset{t \rightarrow 0}{\lim }\, \frac{1}{2} \int  \frac{G(x-y + tv(x) ) - G(x-y)}{t} d\mu(x) d\nu(y) \\
    &= \underset{t \rightarrow 0}{\lim }\,  \int_{x \neq y}  \frac{G(x-y + t(v(x) - v(y))) - G(x-y)}{t} d\mu(x) d\mu(y) \\
    &= \int_{x \neq y}  \nabla G (x-y) \cdot v(x) d\mu(x) d\nu(y) \\
     &= \int_{x \neq y}  \nabla G (x-y) \cdot v(x) d\mu(x) d\nu(y) \,.
\end{align*}
We conclude that Formula \eqref{lim deriv direc} holds for the Energy Distance kernel.

\vspace{5mm}

\noindent
\textbf{Second case: } If $G = \frac{1}{\|.\|^{d-2}}$.
Here, we need to use the result that if a positive measure has finite Coulomb energy, then it cannot be too singular, i.e. if we define the diagonal $\mathcal{D} \coloneqq \{x=y\} \subset X \times X $ then $\mu^{ \otimes 2}( \mathcal{D} )= 0$.
Indeed, $\int_\mathcal{D} G d\mu^{ \otimes 2} \leq \int Gd\mu^{ \otimes 2} < \infty$, and $G = +\infty$ on $\mathcal{D}$, which proves $\mu^{ \otimes 2}(\mathcal{D})=0$.
This means that for any measure $\mu$ with finite energy: 
\begin{equation*}
    \int G d \mu^{ \otimes 2} = \int_{\mathcal{D}^c} G d \mu^{ \otimes 2} \,.
\end{equation*}

\noindent
Thus we can conclude that Formula \eqref{lim deriv direc} holds for the Coulomb kernel with similar computations, as $G$ is $\mathcal{C}^\infty$ on $\mathcal{D}^c$.
We write: 
$$\frac{E_\nu((id + t v)\#\mu) - E_\nu(\mu)}{t} = \frac{E_\nu((id + tv)\#\mu) - E_\nu(\mu)}{W_2((id+tv)\#\mu,\mu)} \frac{W_2((id+tv)\#\mu,\mu)}{t}\,.$$

\noindent
By definition of the slope, we get: 
\begin{equation*}
    \underset{t \rightarrow 0}{\limsup }\, \frac{E_\nu((id + tv)\#\mu) - E_\nu(\mu)}{W_2((id+tv)\#\mu,\mu)} \leq |\partial E_\nu | (\mu) \,.
\end{equation*}
Additionally,
\begin{equation*}
        \underset{t \rightarrow 0}{\limsup }\, \frac{W_2((id+tv)\#\mu,\mu)}{t}  \leq \|v \|_{L_2(\mu)}\,,
\end{equation*}
so that, taking the limit for $t\rightarrow 0$ we get:
\begin{equation*}
    \int_{x \neq y} (\nabla G) \star (\mu - \nu) \cdot v d\mu \leq |\partial E_\nu | (\mu) \|v \|_{L_2(\mu)} \, .
\end{equation*}
\noindent
Similarly, using $-v$ instead of $v$, we have:
\begin{equation*}
    \left|\int_{x \neq y} (\nabla G) \star (\mu - \nu) \cdot v d\mu \right| \leq |\partial E_\nu | (\mu) \|v \|_{L_2(\mu)}\,.
\end{equation*}
As the vector field $v$ chosen is arbitrary, we get: 
\begin{equation*}
    \|(\nabla G) \star (\mu - \nu) \|_{L_2(\mu)} \leq |\partial E_\nu | (\mu) \, ,
\end{equation*}
which proves the result.
\end{proof}

\begin{proof}[Proof of Lemma \ref{DiffRegMeas}]
    Let $\rho$ be a measure, and $id + v$ an optimal transport between $\mu$ and $\nu$ (so that $\rho = (id + v)\#\mu$) which exists as $\mu \in \mathcal{P}_2^r$. We note $\xi := \nabla G \star(\mu - \nu)$, which is a well-defined vector field everywhere.
    Using the same computations as in the proof of equality \eqref{lim deriv direc}, we get:
    \begin{equation*}
        E_\nu(\rho) - E_\nu(\mu) = \int \xi \cdot v d\mu + o(\|v\|_{L_2(\mu)}^2).
    \end{equation*}
    However, as $v$ is an arbitrary optimal transport plan between $\mu$ and $\rho$ and $\mu$ is regular, we get:
    \begin{equation*}
        E_\nu(\rho) - E_\nu(\mu) \geq \underset{\gamma_0 \in \Gamma_0(\mu,\rho)}{\inf} \, \int \xi \cdot (y-x) d\gamma_0 + o(W_2^2(\mu,\rho)) \,,
    \end{equation*}
    which concludes the proof by Definition \ref{VecSubDiff}.
\end{proof}

\end{document}